\documentclass[11pt,letterpaper]{article}
\usepackage{amsthm,amsmath,amsfonts,amssymb}
\usepackage{tikz}
\usetikzlibrary{calc}
\usetikzlibrary{intersections,angles,quotes,decorations.pathreplacing,arrows.meta}
\usepackage{authblk}
\usepackage[margin=1in]{geometry}
\usepackage{enumerate}
\usepackage{hyperref}
\usepackage[dvipsnames,svgnames]{xcolor}

\DeclareMathOperator{\Vol}{Vol}
\DeclareMathOperator{\Exp}{Exp}
\DeclareMathOperator{\Hess}{Hess}
\NewDocumentCommand{\sff}{}{\mathrm{I\!I}}
\newcommand{\Ric}{\operatorname{Ric}}

\newcommand{\RR}{\mathbb R}
\newcommand{\CC}{\mathrm C}
\newcommand{\JJ}{\mathrm J}
\newcommand{\VV}{\mathrm V}
\newcommand{\HH}{\mathrm H}

\newcommand{\eps}{\varepsilon}
\newcommand{\vphi}{\varphi}
\newcommand{\cH}{\mathcal H}
\newcommand{\cV}{\mathcal V}

\newtheorem{theorem}{Theorem}[section]
\newtheorem{Theorem}{Theorem}
\newtheorem{definition}[theorem]{Definition}
\newtheorem{lemma}[theorem]{Lemma}

\newtheorem{proposition}[theorem]{Proposition}
\newtheorem*{proposition*}{Proposition}

\newtheorem*{remark*}{Remark}

\title{Illumination bodies on Riemannian manifolds}
\author[1]{Rotem Assouline}
\affil[1]{\small Sorbonne Université, Université Paris Cité, CNRS, IMJ-PRG, F-75005 Paris, France}
\author[2]{Carsten Schütt}
\affil[2]{\small University of Kiel, 24118 Kiel, Germany}
\author[3]{Elisabeth M. Werner}
\affil[3]{\small Case Western Reserve University, Cleveland, Ohio 44106, USA}
\date{}

\begin{document}

    \maketitle

    \begin{abstract}
        We prove a generalization of Werner's asymptotic formula for the volume of the illumination body of a convex body, which holds on Riemannian manifolds with Ricci curvature bounded from below. The $\delta$-illumination body of a subset of a Riemannian manifold is defined to be the set of all points such that the union of all minimizing geodesic segments joining the point to the set has volume at most $\delta$. 
    \end{abstract}

    \section{Introduction}

    For a convex body $K \subseteq \RR^n$ and $\delta > 0$, the $\delta$-illumination body of $K$ is defined by
    \[
        K^\delta
        =
        \left\{
            x \in \mathbb{R}^n
            :
            \Vol(\mathrm{conv}(K\cup\{x\})\setminus K) \le \delta
        \right\},
    \]
    where $\Vol$ denotes volume and $\mathrm{conv}$ denotes convex hull.
    This construction was introduced by Werner in \cite{Wer} and is, in a sense, dual to the convex floating body introduced  independently by Bárány and Larman \cite{BaranyLarman} and Sch\"utt and Werner
    \cite{SW} (earlier notions of a floating body appear in the works of Blaschke \cite{Blaschke} and Dupin \cite{Dupin}).

    \medskip
    It was shown in \cite{SW,Wer} that the leading order in the volume expansions for both the floating body and the illumination body as $\delta \searrow 0$ is, up to dimensional constants, the \emph{affine surface area} of the body, which is defined by
    \begin{equation}\label{eq:affine-surface-area}
        \int_{\partial K}\kappa^{\frac{1}{n+1}}d\cH^{n-1}
    \end{equation}
    where $\kappa$ is the (generalized) Gauss curvature and $\cH^{n-1}$ is the $(n-1)$-dimensional Hausdorff measure.
    These results have inspired numerous generalizations, including 
    $L_p$-, Orlicz, and stochastic versions
    \cite{Lutwak, Reitzner, SW5, TW:2023, Ye:2015, Ye:2016, Zhao2016}.   

    \medskip
    Besau and Werner \cite{BW:2016, BW} defined the floating body of a convex body in a space of constant curvature, and proved that the leading term in the expansion for its volume is again, up to the same dimensional constant, the integral \eqref{eq:affine-surface-area}, with $\kappa$ being the Gauss curvature of $\partial K$ with respect to the ambient Riemannian metric. The analogous result for illumination bodies in spaces of constant curvature was established recently by Assouline, Besau and Werner \cite{ABW}.

    \medskip
    Due to the lack of a natural notion of hyperplanes, it is not clear how to define floating bodies on Riemannian manifolds of nonconstant curvature. In contrast, the illumination body does admit a straightforward generalization to arbitrary Riemannian manifolds. The goal of this paper is to introduce such a generalization and to show that Werner's asymptotic formula remains valid in this setting, with the volume and Gauss curvature replaced by their intrinsic Riemannian counterparts.

    \medskip
    Let $(M,g)$ be a Riemannian manifold and let $K \subseteq M$. For a point $x \in M \setminus K$, we denote by $\CC(x,K)$ the union of all minimizing geodesics which join the point $x$ to the set $K$ and whose relative interiors lie in $M\setminus K$:
    \[
        \CC(x,K) : = \bigcup
        \left\{\,\,
        \gamma([0,1])
        \quad \middle| \quad 
        \begin{array}{c}
            \gamma\,\,\text{is a minimizing geodesic}, \,\,\gamma(0) = x,\vspace{.1cm}\\   \gamma(1) \in K, \quad \gamma(s) \in M \setminus K \, \,\,\forall\,s \in [0,1)
        \end{array}
        \right\},
    \]
    see Figure \ref{Cxkfig}. 

    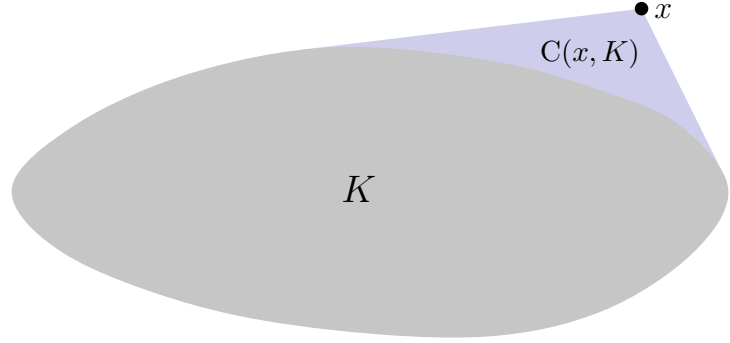
\begin{figure}[h]
        \centering
        \begin{tikzpicture}[scale=1.7]

            \definecolor{lightpurple}{RGB}{190,190,230}
            \definecolor{darkgray}{RGB}{120,120,120}
            \definecolor{lightlavender}{RGB}{218,212,242}

            \coordinate (x) at (3.2,1.4);


            \fill[lightpurple, opacity=0.75] 
            (x) -- (0.55,1.08) --(3.865,0.056)-- cycle;
            
            \fill[gray!45] plot[smooth cycle, tension=0.6] coordinates {
            (-1.7,0) (-1.2,0.5) (-0.3,0.9) (0.8,1.1) (2,1) (2.9,0.76) 
            (3.5,0.49) (3.875,0) (3.6,-0.5) (2.9,-0.95) (2,-1.15) (0.8,-1.1) 
            (-0.3,-0.9) (-1.3,-0.5)
            };

            \node at (1,0) {\Large $K$};

            \fill (x) circle (1.5pt);
            \node[below=1, right=1] at (x) {\large $x$};

            \node at (2.8,1.05) {\normalsize $\CC(x,K)$};

        \end{tikzpicture}
        \caption{The set $\CC(x,K)$.}
        \label{Cxkfig}
    \end{figure}

    Denote by $\Vol$ the Riemannian volume measure on $M$.
    \begin{definition}[Illumination body]\normalfont
        Let $\delta > 0$. The \emph{$\delta$-illumination body} of $K$ is the set
        \begin{equation}\label{eq:Kdeltadef}
            K^\delta : = \left\{x \in M\setminus K : \Vol(\CC(x,K)) \le \delta\right\} \cup K.
        \end{equation}
    \end{definition}

    Our main result generalizes to the Riemannian setting the asymptotic formula established in \cite{Wer} for the volume of the illumination body of a convex body.
    Since there are several notions of convexity in Riemannian geometry, 
    and terminology may vary between sources, we specify precisely what 
    kind of convexity we require from the set $K$:
    \begin{equation*}\label{eq:Kcondition}
        \text{every minimizing geodesic with endpoints in $K$ is contained entirely in $K$.}\tag{$\star$}
    \end{equation*}
    
    We also require that $K$ be \emph{full-dimensional}, meaning that $K$ is the closure of its interior. The manifold $M$ itself is required to be \emph{geodesically-convex}, i.e. every pair of points in $M$
    can be joined by a minimizing geodesic (which is the case in particular if $M$ is complete and connected).

    \begin{Theorem}\label{mainthm}
        Let $(M,g)$ be a smooth, oriented, geodesically-convex, $n$-dimensional Riemannian manifold with Ricci curvature bounded from below and let $K \subseteq M$ be a compact, full-dimensional set satisfying \eqref{eq:Kcondition}. Then
        \begin{equation}\label{maineq}
            \lim_{\delta\searrow 0}\frac{\Vol\left(K^\delta\setminus K\right)}{\delta^{\frac{2}{n+1}}} = \beta_n\cdot \int_{\partial K}\kappa^{\frac{1}{n+1}}d\cH^{n-1},
        \end{equation}
        where $\cH^{n-1}$ is the $(n-1)$-dimensional Hausdorff measure, $\kappa$ is the (generalized) Gauss curvature of $\partial K$, and the constant $\beta_n$ is given by
        \[
            \beta_n : = \frac12\cdot \left(\frac{n(n+1)}{\vartheta_{n-1}}\right)^{\frac{2}{n+1}},
        \]
        where $\vartheta_{n-1}$ is the volume of the unit ball in $\RR^{n-1}$.
    \end{Theorem}
    \vskip 2mm

    Condition \eqref{eq:Kcondition} cannot be relaxed to mere geodesic convexity of $K$ without extra assumptions;
    indeed, if $K$ is a square of sidelength $1/2$ with sides parallel to the axes in the flat torus $M = \RR^2/\mathbb{Z}^2$ , then $K^\delta = K$ for all $\delta$ sufficiently small. Note that, since condition \eqref{eq:Kcondition}
    only pertains to minimizing geodesics, it can be satisfied by subsets of compact manifolds.
    For instance, if $M$ is the round sphere and $K$ is a geodesically-convex set
    contained in an open hemisphere, then $K$ satisfies \eqref{eq:Kcondition}. The assumption that the Ricci curvature of $M$ is bounded from below is used in order to show that $K^\delta\to K$ in the Hausdorff metric as $\delta\searrow 0$, see Lemma \ref{lem:Hausdorff}.

    \medskip
    The overall scheme of our proof of Theorem \ref{mainthm} follows loosely that of the original proof given in \cite{Wer}, see also \cite{SW}. 
    However, the details are substantially different, since we do not have at our disposal the tools of Euclidean convex geometry. 
    We work locally, in normal coordinates about a point close to the boundary, and rely on fine analysis of the signed distance function to the set $K$. 
    Instead of using polar coordinates to estimate the volume of $K^\delta$, we employ a tube formula for integration over small neighborhoods of the set $K$ (Proposition \ref{tubeprop}). 
    We therefore obtain an alternative proof of the theorem in the Euclidean case as well.

    \medskip
    The proof is divided into three steps: in the first, we prove a pointwise asymptotic formula for $\Vol(\CC(x,K))$ when the boundary of $K$ is twice differentiable at the nearest point to $x$; in the second, we establish a uniform lower bound on $\Vol(\CC(x,K))$, which in turn implies a uniform upper bound on the integrand in the tube formula; and in the third step, we apply the dominated convergence theorem and conclude the proof. 

    \medskip
    In constant curvature, there exists a global coordinate chart in which convex bodies are also convex in the Euclidean sense; this fact was used in \cite{ABW, BW} in order to reduce to the case of weighted floating and illumination bodies in Euclidean space. In the general case no such reduction is possible. That said, some lemmas concerning regularity of the boundary (Lemma \ref{lem:alexandrov}, Lemma \ref{rhointegrablelemma}) can be reduced to their Euclidean counterparts thanks to the existence of a \emph{local} chart about  each boundary point, in which the boundary is convex in the Euclidean sense. The latter property is enjoyed by any set of positive reach \cite{Lyt}.

    \medskip
    In Euclidean space (as well as in hyperbolic space and on the round hemisphere), $\CC(x,K) \cup K = \mathrm{conv}(\{x\}\cup K)$. We do not know
    whether Theorem \ref{mainthm} remains valid with $\CC(x,K)$ replaced by
    $\mathrm{conv}(\{x\}\cup K)\setminus K$ in the definition of $K^\delta$ (say, when $M$ is a Cartan--Hadamard manifold,
    so that the convex hull is well-defined).

    \medskip
    The rest of the paper is structured as follows. In Section \ref{convsec}, we recall some basic facts about convex sets in Riemannian manifolds. In Section \ref{proofsec} we prove Theorem \ref{mainthm}. In the Appendix we gather some standard comparison results, and prove a technical lemma on second-order differentiability of the signed distance function at regular boundary points.

    \medskip
    \textbf{Acknowledgments}.
    The authors want to thank the Hausdorff Institute (Bonn, Germany).
    It was during the 2024 trimester ``Synergies between modern probability,  geometric analysis and stochastic geometry'' that the project got started.
    Elisabeth M. Werner was supported by NSF grant DMS-2506790. Rotem Assouline was supported by the Fondation Mathématique Jacques Hadamard (FSMP) and by a Rothschild fellowship (Yad Hanadiv).

    \section{Convex subsets of Riemannian manifolds}\label{convsec}

    \subsection{Basic definitions and facts}\label{basicsec}

    Let $(M,g)$ be a smooth, oriented, $n$-dimensional Riemannian manifold ($n \ge 2$). The inner product, norm, distance function and volume measure associated to $g$ will be denoted by $\left<\cdot,\cdot\right>,\,|\cdot|,\,d$ and $\Vol$, respectively. 
    
    \medskip
    A set $K \subseteq M$ is called \emph{(geodesically) convex} if for every $x,y \in K$ there exists a minimizing geodesic joining $x$ to $y$ which is contained in $K$, and \emph{strongly convex} if for every $x,y \in K$ there exists a unique minimizing geodesic joining $x$ to $y$, and that geodesic is contained in $K$.

    \medskip
    We shall assume that the manifold $M$ itself is geodesically-convex, and that the Ricci curvature of $M$ is bounded from below, i.e. there exists a constant $k \in \RR$ such that
    \[
        \Ric(v,v) \ge k \cdot|v|^2 \qquad \text{for all $v \in TM$.}
    \]

    Throughout the paper, we fix a compact set $K \subseteq M$ satisfying \eqref{eq:Kcondition} which is assumed to be \emph{full-dimensional}, meaning that it is the closure of its interior, and write
    \[
        \Sigma : = \partial K.
    \]

    We denote by $r : M \to \RR$ the signed distance function to $\Sigma$:
    \[
    r(x) : = 
    \begin{cases}
        d(x,K) & x \in M \setminus K\\
        -d(x,M\setminus K) & x \in K
    \end{cases}
    = 
    \begin{cases}
        d(x,\Sigma) & x \in M \setminus K\\
        -d(x,\Sigma) & x \in K,
    \end{cases}
    \]
    and set
    \[
    {U_t} : = \{x \in M : r(x) < t\}, \qquad t \in \RR
    \]
    and 
    \[
    \Sigma_t : = \left\{x \in M : r(x) = t\right\}, \qquad t \in \RR.
    \]

    We also fix $R>0$, which we shall take to be as small as we wish, as long as it only depends on the set $K$, and define
    \[
        {U} : = U_R = \{x \in M : r(x) < R\}.
    \]
    For the time being, we require $R$ to be small enough that:

    \begin{itemize}
        \item[(i)] For every $x \in {U}$ there exists a unique point $\Pi x \in K$ minimizing the distance to the point $x$ among the points in $K$.
        \item[(ii)] The function $r$ is $C^{1,1}$ on the set ${U}\setminus K$.
        \item[(iii)] For every $x \in {U}$, the ball $B_{4R}(x)$ is strongly convex and $\exp_x\vert_{B_{4R}(0)}$ is a diffeomorphism onto $B_{4R}(x)$.
    \end{itemize}

    That a positive number $R$ with the first two properties exists was established in \cite{Wal74,Wal}; indeed, condition \eqref{eq:Kcondition} implies that $K$ is locally convex in the sense of \cite{Wal74,Wal}. Property (iii) can be guaranteed by Whitehead's theorem \cite[Theorem 5.14]{CE} and by compactness.

    \medskip
    The \emph{tangent cone} to $K$ at a point $x \in  \Sigma$ is the set of nonzero tangent vectors $v \in T_xM$ such that some geodesic segment with initial velocity $v$ is contained in the interior of $K$:
    \[
        T_xK : = \left\{v \in T_xM\setminus \{0\} \quad \middle| \quad \exists \,t > 0 : \quad \exp_x(sv)\in \mathrm{int} K \quad \forall s \in (0,t]\right\}.
    \]

    \begin{lemma}[\text{\cite[Proposition 1.8]{CG}}]\label{tangentconelemma}
        For every $x \in \Sigma$,  the tangent cone $T_xK$ is a nonempty, open convex cone.
    \end{lemma}

    \begin{lemma}\label{lem:convexchart}
        For every $x \in \Sigma$ there exists a neighborhood $V_x \ni x$ and a coordinate chart on the set $V_x$ in which $K \cap V_x$ is a convex set with respect to the Euclidean metric.
    \end{lemma}
    \begin{proof}
        Being compact and convex, the set $K$ has positive reach, i.e. there exists a uniform neighborhood of $K$ on which the nearest-point projection onto $K$ is single valued, see property (i) above. Let $x \in \Sigma$. By Lemma \ref{tangentconelemma}, the tangent cone $T_xK$ is open, i.e. has full dimension $n$. Hence, by \cite[Section 5.2]{Lyt}, there exists a neighborhood $V_x \ni x$ and a smooth coordinate chart on $V_x$ in which $K \cap V_x$ is convex in the Euclidean sense (the main theorem in \cite{Lyt} asserts the existence of a $C^{1,1}$ chart with this property, but the proof in the full-dimensional case given in Section 5.2 provides a smooth chart).
    \end{proof}

    Let $SM$ denote the unit tangent bundle of $M$:
    \[
        SM : = \{v \in TM : |v| = 1\}.
    \]
    A unit tangent vector $u \in S_xM$ at a point $x \in \Sigma$ is called an \emph{outer normal to $K$ at $x$} if its inner product with every vector in the tangent cone $T_xK$ is nonpositive. Thus the set of all outer unit normals to $K$ at $x$, which we denote by
    \[
        N_x : = \left\{u \in S_xM : \left<u,v\right> \le 0\quad \forall v \in T_xK\right\},
    \]
    is the intersection of the unit sphere $S_xM$ with the polar cone to the tangent cone $T_xK$. 

    \medskip
    Let ${\Sigma'}$ denote the set of boundary points $x \in \Sigma$ for which the set $N_x$ is a singleton, and for $x \in \Sigma'$ denote the unique element of $N_x$ by $\nu\vert_x$. Thus
    \[
        \Sigma' : = \left\{x \in \Sigma : \#N_x = 1\right\} = \left\{x \in \Sigma : N_x = \{\nu\vert_x\}\right\}.
    \]

    If $r$ is differentiable at a point $x \in \Sigma$, then $x \in \Sigma'$ and 
    \begin{equation}\label{gradienteq}
        \nabla r\vert_x = \nu\vert_x.
    \end{equation}
    
    Indeed, if $u \in N_x$ then $r(\exp_x(su)) = s$ for all $s \in [0,R)$, see \cite{Wal74}, so $dr\vert_x(u) = 1$; since $r$ is $1$-Lipschitz, $|\nabla r\vert_x| \le 1$, which forces $\nabla r\vert_x = u$. 

    \medskip
    For $x \in \Sigma'$, we denote by $\gamma_x$ the \emph{perpendicular to $K$ at $x$}, which is the unit-speed geodesic
    \begin{equation}\label{eq:perpendicular}
        \gamma_x(t) : = \exp_x(t\cdot \nu\vert_x), \qquad -4R < t < 4R.
    \end{equation}

    \subsection{Second fundamental form and Gauss curvature}\label{sec:sff}

    A function $f:M\to \RR$ is said to be \emph{differentiable twice in the sense of Alexandrov} at a point $x \in M$ if it is differentiable at $x$ and there exists a quadratic form $\HH$ on $T_xM$ such that as $v \to 0$ in $T_xM$,
    \begin{equation}\label{alexandroveq0}
        f(\exp_x(v)) = f(x) + df\vert_x(v) + \frac12\cdot \HH(v) + o(|v|^2).
    \end{equation}
    In this case we denote the quadratic form $\HH$ by $\Hess f\vert_x$ and call it the \emph{Hessian} of the function $f$ at the point $x$. 

    \medskip
    The boundary of a convex body in $\RR^n$ is locally the graph of a convex function, which by Alexandrov's theorem (see e.g. \cite{BCP} or \cite{EG}) is twice differentiable almost everywhere in the sense of Alexandrov. We need a similar result in our setting, which we formulate using the signed distance function. We denote by 
    \[
        \Sigma'' \subseteq \Sigma
    \]
    the set of points where the signed distance function $r$ is differentiable twice in the sense of Alexandrov. 

    \begin{lemma}\label{lem:alexandrov}
        $\cH^{n-1}(\Sigma \setminus \Sigma'') = 0$.
    \end{lemma}
    \begin{proof}
        By compactness, it suffices to prove that for every $x \in \Sigma$ there exists a neighborhood $V \ni x$ such that for $\cH^{n-1}$-almost every $y \in \Sigma \cap V$, the function $r$ is twice differentiable in the sense of Alexandrov at the point $y$.
        
        \medskip
        Let $x \in \Sigma$. By Lemma \ref{lem:convexchart}, we can find a neighborhood $V \ni x$ admitting a coordinate chart in which $K \cap V$ is convex \emph{in the Euclidean sense}. By Alexandrov's theorem and the fact that the Euclidean and Riemannian $(n-1)$-dimensional Hausdorff measures on $\Sigma \cap V$ are mutually absolutely continuous in this chart, for $\cH^{n-1}$-almost every $y \in \Sigma\cap V$ there exists a neighborhood $V_y \ni y$ and a function $\psi : V_y \to \RR$ such that $K \cap V_y = \psi^{-1}((-\infty, 0])$, and such that $\psi$ is differentiable twice at $y$ in the sense of Alexandrov, and $d\psi\vert_y  \ne 0$. By Lemma~\ref{lem:sdf}, the signed distance function $r$ is differentiable twice in the sense of Alexandrov at every such point $y$.
    \end{proof}
    By \eqref{gradienteq} and \eqref{alexandroveq0}, for every $x \in \Sigma''$,

    \begin{equation}\label{alexandroveq}
        r(\exp_x(w)) = \left<w,\nu\vert_x\right> + \frac12\cdot\Hess r\vert_x(w) + o(|w|^2) \qquad \text{ as} \qquad T_xM \ni w \to 0.
    \end{equation}

    The expansion \eqref{alexandroveq} and convexity of $K$ imply that $\Hess r\vert_x$ is positive semidefinite. Moreover, since $r(\exp_x(t\nu\vert_x))=t$ for $t \in [0,R)$, the normal $\nu\vert_x$ lies in the kernel of $\Hess r\vert_x$. The quadratic form 
    \[
        \sff\vert_x : = \Hess r\vert_{\nu\vert_x^\perp},
    \]
    i.e., the restriction of the Hessian of $r$ to the orthogonal complement of the normal $\nu\vert_x$ in $T_xM$, is called the \emph{second fundamental form} of $\Sigma$ at $x$.  By Lemma \ref{lem:alexandrov}, the second fundamental form is defined $\cH^{n-1}$-almost everywhere on $\Sigma$. Its eigenvalues $\kappa_1,\dots,\kappa_{n-1}$ are called the \emph{principal curvatures} of $\Sigma$, and their product
    \[
        \kappa : = \kappa_1\cdots\kappa_{n-1} = \det\sff
    \]
    is called the \emph{Gauss curvature} of $\Sigma$,  which is a nonnegative, Borel measurable function on $\Sigma$ (measurability follows from the fact that $\Hess r$ is the absolutely continuous part of the distributional Hessian of $r$, see e.g. \cite{EG}).

    \begin{remark*}\normalfont
        A more familiar definition of the second fundamental form is
        \[
            \sff(w) : = \left<\nabla_w\nu,w\right>, \qquad w \in T\Sigma,
        \]
        where $\nabla_w\nu$ is the covariant derivative of the outer unit normal $\nu$ in the direction $w$. Since the outer unit normal is the gradient of the signed distance function, this definition indeed coincides with ours (certainly when $\Sigma$ is $C^2$, and also in general with the appropriate definition of a covariant derivative, see e.g. \cite[Section 14]{Vil}).
    \end{remark*}

    \subsection{Integration over tubes}

    The following formula is classical when $M = \RR^n$  \cite[Chapter 4]{Sc} and for general $M$ when $\partial K$ is smooth \cite{We,Sa,Gr} (see also \cite{ST}). The next (and more interesting) terms of the expansion involve curvature measures and will not be needed here. We denote the $k$-th dimensional Hausdorff measure by $\cH^k$.

    \begin{proposition}\label{tubeprop}
        For every $0 < t < R$ and every Borel function $f:M \to [0,1]$,
        \begin{align}\label{tubeq}
            \int_{U_t}fd\cH^n & = \int_Kfd\cH^n +  \int_{\Sigma'}\left(\int_0^tf(\gamma_x(s))ds\right)d\cH^{n-1}(x) + O(t^2),
        \end{align}
        where the implied constant depends only on the set $K$.
    \end{proposition}

    In order to prove Proposition \ref{tubeprop} we need a short preparation. Let $\{\Phi_t\}_{t \in \RR}$ denote the geodesic flow on the unit tangent bundle $SM$. For all $t \in \RR$, let
    \[
        \cV_t : = \left\{v \in SM : \Phi_sv \text{ is defined for all $s \in [0,t]$}\right\}
    \]
    and define
    \[
        \Exp_t:\cV_t \to M, \qquad v \mapsto \exp(tv) = \pi(\Phi_tv),
    \]
    where 
    \[
        \pi : TM \to M
    \]
    is the natural projection. Note that $\bigcup_{t > 0}\cV_t = \bigcup_{t < 0} \cV_t = SM$ and that
    \[
        \Exp_0 = \pi.
    \]
    The \emph{normal cycle} of the set $K$ is the set of all normal vectors at its boundary points:
    \[
        N : = \bigcup_{x \in \Sigma}N_x \subseteq SM.
    \]
    By our choice of $R$, the normal cycle $N$ is contained in the set $\cV_{4R}$. Set also
        \[
            N' : = N\cap\pi^{-1}(\Sigma') = \left\{\nu\vert_x : x \in \Sigma'\right\}.
        \]

    \begin{theorem}[\text{\cite{Wal}, see also \cite[Proposition 4.3]{ST}}]\label{normalcycletheorem}
        The normal cycle $N$ is a compact $(n-1)$-dimensional oriented Lipschitz submanifold of the unit tangent bundle $SM$. For every $t \in (0,R)$, the map $\Exp_t$ restricts to an orientation-preserving bi-Lipschitz homeomorphism from $N$ to $\Sigma_t$.
    \end{theorem}

    \begin{lemma}[\text{\cite[Theorem 1]{Wal74}}]\label{perplemma}
        For every $v \in N$ and every $t \in (0,R)$,
        \[
            r(\Exp_t(v)) = t \qquad \text{ and } \qquad \nabla r\vert_{\Exp_t(v)} = \Phi_tv.
        \]
    \end{lemma}

    Let $\omega$ denote the Riemannian volume form on the manifold $M$. For each $t \in \RR$, define an $(n-1)$-form $\tilde\omega_t$ on $\cV_t$ by:
    \[
        \tilde\omega_t(\xi_1,\dots,\xi_{n-1}) : = \omega(\Phi_tv,d\Exp_t(\xi_1),\dots,d\Exp_t(\xi_{n-1})), \qquad \xi_1,\dots,\xi_{n-1} \in T_vSM, \, \, v \in \cV_t.
    \]
    \begin{lemma}\label{COVlemma}
        Let $t \in (0,R)$ and let $f : \Sigma_t \to \RR$ be an $\cH^{n-1}$-integrable function. Then
        \[
            \int_{\Sigma_t}f\,d\cH^{n-1} = \int_N(f\circ\Exp_t)\,\tilde\omega_t.
        \]
        Moreover, if $f:\Sigma \to \RR$ is an $\cH^{n-1}$-integrable function then
        \[
            \int_{\Sigma'} f\,d\cH^{n-1}=\int_{N'}(f\circ\pi)\,\tilde\omega_0.
        \]
    \end{lemma}
    \begin{proof}
        Let $t \in (0,R)$. By Theorem \ref{normalcycletheorem}, the normal cycle $N$ is a Lipschitz submanifold, which is mapped Lipschitz-continuously by $\Exp_t$ to the level set $\Sigma_t$. Let $v \in N$ be a point such that the tangent space $T_vN$ exists, and let $\xi_1,\dots,\xi_{n-1}$ be a basis for $T_vN$ such that $\tilde\omega_t(\xi_1,\dots,\xi_{n-1}) = 1$. Then by Lemma \ref{perplemma}, 
        \begin{align*}
            \tilde\omega_t(\xi_1,\dots,\xi_{n-1}) & = \omega\left(\Phi_tv,d\Exp_t(\xi_1),\dots,d\Exp_t(\xi_{n-1})\right)\\
            & = \omega\left(\nabla r\vert_{\Exp_tv},d\Exp_t(\xi_1),\dots,d\Exp_t(\xi_{n-1})\right)\\
            & = \left(\Exp_t^*(\iota_{\nabla r}\omega)\right)(\xi_1,\dots,\xi_{n-1}),
        \end{align*}
        where $\iota$ denotes interior product. Since $\nabla r$ is the outer unit normal to $\Sigma_t$, the volume form $\iota_{\nabla r}\omega$ induces the restriction of the Hausdorff measure $\cH^{n-1}$ to the Lipschitz hypersurface $\Sigma_t$. 
        
        \medskip
        For $v \in N'$ we have $v=\nu\vert_{\pi(v)}$, and hence by the defintion of $\tilde \omega_0$ and the fact that $\Exp_0 = \pi$ and $\Phi_0 = \mathrm{Id}$,
        \[
            \tilde\omega_0 = \pi^*(\iota_{\nu}\omega) \qquad \text{on $N'$}.
        \]
        
        In both cases, the desired formulas now follow from the change-of-variables formula for Lipschitz maps, see e.g. \cite[Chapter 3]{EG}.
    \end{proof}

    \begin{lemma}\label{tildeomegatlemma}
        We have that
        \[
            \sup\left\{\int_Nf\,(\tilde\omega_t - \tilde\omega_0) \quad \middle| \quad f : N \to \RR \text{ Borel }, \, \, |f| \le 1\right\} = O(t),
        \]
        where the implied constant depends only on the set $K$.
    \end{lemma}

    \begin{proof}
        By our assumptions, $N \subseteq \cV_{4R}\cap \cV_{-4R}$, and therefore, by smoothness of the geodesic flow $\Phi_t$, the $(n-1)$-form $\tilde\omega_t$ is well-defined and smooth in a neighborhood of $N$ for all sufficiently small $t$, and depends smoothly on the parameter $t$. The lemma then follows from compactness of $N$.
    \end{proof}

    \begin{proof}[Proof of Proposition \ref{tubeprop}]
        By Fubini's theorem and Lemma \ref{COVlemma},
        \begin{align*}
            \int_{\Sigma'}\left(\int_0^tf(\gamma_x(s))ds\right)d\cH^{n-1}(x) & = \int_0^t\left(\int_{\Sigma'}f(\gamma_x(s))d\cH^{n-1}(x)\right)ds\\
            & = \int_0^t\left(\int_{N'}f\left(\gamma_{\pi(\cdot)}(s)\right)\,\tilde\omega_0\right)ds\\
            & = \int_0^t\left(\int_{N'}(f\circ\Exp_s)\,\tilde\omega_0\right)ds,
        \end{align*}
        where the last passage holds true since $\gamma_{\pi(v)}$ is a unit-speed geodesic with initial velocity $v$. Moreover, $\tilde\omega_0=0$ almost everywhere on $N\setminus N'$; indeed, if $N$ has a tangent space at some $v \in N \setminus N'$ then $d\pi$ has a nontrivial kernel at $v$. Hence 
        \begin{equation}\label{intSigmaprime1}
            \int_{\Sigma'}\left(\int_0^tf(\gamma_x(s))ds\right)\,d\cH^{n-1}(x) = \int_0^t\left(\int_N(f\circ\Exp_s)\,\tilde\omega_0\right)ds.
        \end{equation}

        On the other hand, since $r$ is $C^{1,1}$ on ${U} \setminus K$ and $|\nabla r|\equiv 1$, by the coarea formula and Lemma \ref{COVlemma},
        \begin{align*}
            \int_{U_t}fd\cH^n - \int_Kfd\cH^{n} & = \int_0^t\left(\int_{\Sigma_s}fd\cH^{n-1}\right)ds\\
            & = \int_0^t\left(\int_N(f\circ \Exp_s)\,\tilde\omega_s\right)ds.\\
        \end{align*}
        Combining this with \eqref{intSigmaprime1} we see that
        \begin{align*}
            & \left|\int_{U_t}fd\cH^n - \int_Kfd\cH^{n} - \int_{\Sigma'}\left(\int_0^tf(\gamma_x(s))ds\right)\,d\cH^{n-1}(x)\right|\\
            &\qquad  \le \int_0^t \left|\int_N(f\circ\Exp_s)\,\tilde\omega_s - \int_N(f \circ \Exp_s)\,\tilde\omega_0\right|ds\\
            &\qquad  = O(t^2),
        \end{align*}
        where the last step follows from Lemma \ref{tildeomegatlemma} and the fact that $|f| \le 1$.
    \end{proof}

    \section{Proof of Theorem \ref{mainthm}}\label{proofsec}

    Recall that for every $\delta > 0$,
    \[
        K^\delta : = \left\{x \in M\setminus K : \Vol(\CC(x,K)) \le \delta\right\} \cup K,
    \]
    where
    \[
        \CC(x,K) : = \bigcup\left\{\,\,\gamma([0,1])\quad \middle| \quad \begin{array}{c}\gamma\,\,\text{is a minimizing geodesic}, \,\,\gamma(0) = x,\vspace{.1cm}\\   \gamma(1) \in K, \quad \gamma(s) \in M \setminus K \, \,\,\forall\,s \in [0,1)\end{array}\right\},
    \]
    and that our goal is to prove that 
    \[
        \lim_{\delta\searrow 0}\frac{\Vol\left(K^\delta\setminus K\right)}{\delta^{\frac{2}{n+1}}} = \beta_n\cdot \int_{\Sigma}\kappa^{\frac{1}{n+1}}d\cH^{n-1}.
    \]

    Our proof will be carried out in three steps: in the first, we prove estimates for the volume of $\CC(x,K)$ when $\Pi x \in \Sigma''$, i.e. when the closest point to $x$ in $K$ is a point where the Gauss curvature is defined. In the second, following an idea used in \cite{SW,Wer}, we find an integrable majorant for the function $x \mapsto \delta^{-\frac{2}{n+1}}\cH^1(\gamma_x^{-1}(K^\delta\setminus K))$ which involves the radius $\rho(x)$ of the largest ball contained in $K$ which touches the boundary at the point $x$. In the third and final step, we use the tube formula \eqref{tubeq} and apply the dominated convergence theorem.

    \medskip
    Throughout the proof, constants (including those implied by big-O notation) are only allowed to depend on the set $K$, unless stated otherwise. Unnumbered constants ($C, \tilde C, c, c', \ldots$) are used inside the proofs and their value may change from one appearance to another. Numbered constants appear in the statements of lemmas and have a fixed value throughout.

    \medskip
    Recall the notation $U_t : = \{r < t\}$. 

    \begin{lemma}\label{lem:Hausdorff}
        $K^\delta\to K$ in the Hausdorff metric as $\delta\searrow 0$. In fact, there exist $C_0, \delta_0 > 0$ such that
        \[
            K^\delta \subseteq U_{C_0\delta^{1/n}} \qquad \text{for all $0 < \delta < \delta_0$.}
        \]
    \end{lemma}
    \begin{proof}
        Let $x \in K^\delta \setminus K$ and write
        \[
            r_0 : = r(x) = d(x,K).
        \]
        
        Since $M$ is geodesically-convex, there exists an open, star-shaped set $\Omega_x \subseteq T_xM$ such that the map $\exp_x\vert_{\Omega_x}$ is injective and its image $\exp_x(\Omega_x)$ has full measure in $M$, and for every $v \in \Omega_x$, the geodesic $t\mapsto\exp_x(tv)$ is minimizing on $[0,1]$ (see e.g. \cite[Proposition 5.19 and Theorem 10.34]{Lee}; this fact is usually stated for complete manifolds, but remains valid for geodesically-convex ones).

        \medskip
        Write
        \[
            \Lambda_x : = \left\{u \in S_xM\,\,\colon \,\, \exists t > 0 \quad tu \in \Omega_x \,\,\text{and} \, \,\exp_x(tu) \in K \right\}.
        \]
        
        For $u \in \Lambda_x$, denote by $t_-$ (resp. $t_+$) the infimum (resp. supremum) of all $t > 0 $ for which $tu \in \Omega_x$ and $\exp_x(tu) \in K$. 
        Observe that
        \begin{equation}\label{eq:tminustplus}
            t_- \ge r_0 \qquad \text{ and } \qquad t_+ \le r_0 + D,
        \end{equation}
        where $D$ denotes the diameter of the set $K$. Note also that
        \begin{equation}\label{eq:CCxKpolar}
            \exp_x\left(\left\{ru \,\,\colon\,\,u \in \Lambda_x,\,\,r \in \left[0,t_-(u)\right]\right\}\right) \subseteq \CC(x,K).
        \end{equation}
        
        Denote by $\JJ_x$ the Jacobian of the map 
        \[
            (r,u) \mapsto \exp_x(ru), \qquad (r,u) \in (0,\infty)\times S_xM\,\,\colon\,\,ru \in \Omega_x.
        \]

        Since the Ricci curvature of $M$ is bounded from below, Riccati comparison implies that there exists $C > 0$ such that
        \begin{equation}\label{eq:BG}
            \frac{\int_0^{t'}\,\JJ_x(r,u)\,dr}{\VV(t';C)} \le \frac{\int_0^{t}\,\JJ_x(r,u)\,dr}{\VV(t;C)} \qquad \text{for all $u \in S_xM$ and all $0 < t \le t' \le t_+(u)$,}
        \end{equation}
        where
        \[
            \VV(t;C) : = \int_0^t\sinh(Cr)^{n-1}\,dr, \qquad t \ge 0
        \]
        (see \cite[Theorem III.4.3]{Chavel} or \cite[Theorem 11.19]{Lee}, integration of \eqref{eq:BG} gives the Bishop-Gromov inequality).
        
        \medskip
       Since $\exp_x(\Omega_x)$ has full measure in $M$,
        \[
            \Vol(K) \le \int_{\Lambda_x}\int_{t_-(u)}^{t_+(u)}\,\JJ_x(r,u)\,dr\,d\sigma_x(u),
        \]
        where $\sigma_x$ is the measure on $S_xM$ induced by the metric $g$. Hence, by \eqref{eq:BG}, \eqref{eq:tminustplus} and \eqref{eq:CCxKpolar},
        \begin{align*}
            \Vol(K) 
            & \le
            \int_{\Lambda_x}\int_0^{t_+(u)}\,\JJ_x(r,u)\,dr\,d\sigma_x(u)
            -
            \int_{\Lambda_x}\int_0^{t_-(u)}\,\JJ_x(r,u)\,dr\,d\sigma_x(u)
            \\
            & \le 
            \int_{\Lambda_x}\frac{\VV(t_+(u);C)}{\VV(t_-(u);C)}\int_{0}^{t_-(u)}\,\JJ_x(r,u)\,dr\,d\sigma_x(u)
            -
            \int_{\Lambda_x}\int_0^{t_-(u)}\,\JJ_x(r,u)\,dr\,d\sigma_x(u)
            \\
            & \le 
            \left(\frac{\VV(r_0 + D;C)}{\VV(r_0;C)}-1\right)\int_{\Lambda_x}\int_0^{t_-(u)}\,\JJ_x(r,u)\,dr\,d\sigma_x(u)
            \\
            & \le
            \left(\frac{\VV(r_0 + D;C)}{\VV(r_0;C)}-1\right)\Vol(\CC(x,K))
            \\
            & \le 
            \left(\frac{\VV(r_0 + D;C)}{\VV(r_0;C)}-1\right)\cdot\delta,
        \end{align*}
        where the last inequality holds since $x \in K^\delta\setminus K$. Thus
        \begin{equation}\label{eq:ratiobound}
            1 + \frac{\Vol(K)}{\delta} \le \frac{\VV(r_0+D;C)}{\VV(r_0;C)} .
        \end{equation}
        
        The function $\VV(t+D;C)/\VV(t;C)$ tends to a finite limit as $t \to \infty$ and is continuous on $(0,\infty)$, and therefore its supremum on the interval $[1,\infty)$ is finite. Hence, by \eqref{eq:ratiobound}, there exists $\delta_0 = \delta_0(C,D,\Vol(K))$ such that if $\delta < \delta_0$ then $r_0 \le 1$. The inequality $\sinh(Cr) \ge Cr$ implies that $\VV(r_0;C) \ge \tfrac{C^{n-1}}{n}\,r_0^{\,n}$, while $\VV(r_0 + D;C) \le \VV(1 + D;C)$ for $r_0 \le 1$, so by \eqref{eq:ratiobound} and monotonicity of $\VV$,
        \[
            \frac{\Vol(K)}{\delta} \le \frac{n\,\VV(1 + D;C)}{C^{n-1}} \cdot r_0^{-n}.
        \]

        Thus $r_0 \le C_0\,\delta^{1/n}$ for all $\delta < \delta_0$, with $C_0 := \big(n\,\VV(1 + D;C)/(C^{n-1}\Vol(K))\big)^{1/n}$. Since $x$ is an arbitrary point in $K^\delta\setminus K$, it follows that $K^\delta \subseteq U_{C_0\delta^{1/n}}$. Since $K \subseteq K^\delta$, this proves that $K^\delta \to K$ in the Hausdorff metric as $\delta \searrow 0$.
    \end{proof}

    \subsection{Step I: pointwise convergence}

    Lemma~\ref{lem:Hausdorff} enables us to consider in the first two steps only points which lie in the neighborhood $U := U_R$ fixed in Section~\ref{basicsec}.

    \medskip
    Recall that $\Sigma''\subseteq\Sigma$ denotes the set of points $x_0 \in \Sigma = \partial K$ such that that the signed distance function $r$ to $\Sigma$ is twice differentiable at $x_0$ in the sense of Alexandrov. We write
    \[
        x_t : = \gamma_{x_0}(t) \qquad \text{and} \qquad \dot x_t : = \dot\gamma_{x_0}(t) \in T_{x_t}M, \qquad -4R < t < 4R,
    \]
    where $\gamma_{x_0}$ is the perpendicular to $K$ at $x_0$ defined in \eqref{eq:perpendicular}. In particular, $\dot x_0 = \nu\vert_{x_0}$ is the unit normal to $\Sigma$ at $x_0$.
    Our goal in the present section is to prove the following estimate for $\Vol(\CC(x_t,K))$:

    \begin{proposition}\label{volumeprop}
        For every $x_0 \in \Sigma''$,
        \[
        \lim_{t\searrow 0 } \frac{\Vol(\CC(x_t,K))^{\frac{2}{n+1}}}{t} =
        \begin{cases}
            \beta_n^{-1}\cdot\kappa(x_0)^{-\frac{1}{n+1}} & \kappa(x_0) > 0,\\
            \infty & \kappa(x_0) = 0.
        \end{cases}
        \]
    \end{proposition}

    For the rest of the section, we fix a point
    \[
        x_0 \in \Sigma''.
    \]
    
    Let
    \[
        \kappa_i : = \kappa_i(x_0),\qquad i = 1,\dots,n-1
    \]
    denote the principal curvatures of $\Sigma$ at $x_0$, and let $w_1,\dots,w_n$ be an orthonormal basis of $T_{x_0}M$ such that 
    \[
        w_n = \nu\vert_{x_0}
        \qquad 
        \text{and}
        \qquad 
        \Hess r\vert_{x_0} = \sum_{i=1}^{n-1}\kappa_i\left<\,\cdot\,,w_i\right>^2.
    \]
    Extend the basis $\{w_i\}$ to a parallel orthonormal frame $W_i = W_i(t)$ along the perpendicular $\gamma_{x_0}$. Since $\gamma_{x_0}$ is a geodesic,
    \begin{equation}\label{Wneq}
        W_n(t) = \dot\gamma_{x_0}(t) = \dot x_t \qquad -R < t < R.
    \end{equation}
    The choice of frame $W_i$ gives rise to a family of linear isometries:
    \[
        L_t:\RR^n \to T_{x_t}M, \qquad L_tv = \sum_{i=1}^nv^i\cdot W_i(t),\quad v = (v^1,\dots,v^n)\in\RR^n,
    \]
    and a family of diffeomorphisms:
    \[
        G_t:B' \to B_{4R}(x_t)\subseteq M, \qquad G_t = \exp_{x_t}\circ\, L_t,
    \]
    depending smoothly on $t$, where 
    \[
        B':=B_{4R}(0)\subseteq \RR^n
    \]
    is the ball of radius $4R$ centered at the origin in $\RR^n$. Set also
    \[
        B = B_{2R}(0) \subseteq B'.
    \]
    By our choice of $R$, the ball $B_{4R}(x_t) = G_t(B')$ is strongly convex for all $t\in(-R,R)$. The following lemma follows from Lemma \ref{normallemma}.
    \begin{lemma}\label{Gtlemma}
        For every $0<\ell<R$ and every measurable set $S \subseteq B_\ell(0)\subseteq B'$,
        \[
            \Vol(G_t(S)) = \Vol_{\RR^n}(S) + O(\ell^{n+2}),
        \]
        where $\Vol_{\RR^n}$ is the Lebesgue measure on $\RR^n$.
    \end{lemma}

    For each $t,t' \in (-R,R)$, the map $G_{t'}^{-1}\circ G_t$ restricts to a diffeomorphism from the ball $B$ to an open subset of $B'$. According to the following lemma, this diffeomorphism is translation by the vector $(0,\dots,0,t-t')$, up to lower order terms.

    \begin{lemma}
        For $v \in B$ and $t,t' \in (-R,R)$,
        \begin{equation}\label{translationeq}
            G_{t'}^{-1}(G_t(v)) = v + (t-t') \cdot e_n + O(|t-t'|\cdot|v^\perp|),
        \end{equation}
        where $\,^\perp$ denotes orthogonal projection onto the orthogonal complement of $e_n = (0,\dots,0,1) \in \RR^n$. 
    \end{lemma}

    \begin{proof}
        Fix $t' \in (-R,R)$. For every $v \in B$, define a curve $\eta_v:(-R,R) \to T_{x_{t'}}M$ by
        \[
            \eta_v(t) := \exp_{x_{t'}}^{-1}\exp_{x_t}(L_tv).
        \]
        By the triangle inequality, $\exp_{x_t}(L_tv) \in B_{4R}(x_{t'})$ for all $t {\,\in\,} (-R, R)$, so the curve $\eta_v$ is well defined.

        \medskip
        Suppose first that $v = s\cdot e_n$ for some $s \in (-2R,2R)$. Then, since $L_te_n = W_n(t) =  \dot x_t$ by \eqref{Wneq},
        \[
            \eta_{s\cdot e_n}(t) = \exp_{x_{t'}}^{-1}\exp_{x_t}(s\cdot\dot x_t) = \exp_{x_{t'}}^{-1}(x_{t + s}) = (t+s - t')\cdot \dot x_{t'},
        \]
        and in particular
        \begin{equation}\label{dotetaseneq}
            \dot\eta_{s\cdot e_n} \equiv \dot x_{t'}.
        \end{equation}
        Now let $v \in B$ be arbitrary. By compactness of the closure of $U$ and the fact that $d(x_t,x_{t'}) < 2R$, the map $(v,t) \mapsto \dot\eta_v(t)$ has Lipschitz seminorm bounded by a constant independent of $x_0$, and therefore by \eqref{dotetaseneq},
        \[
            \dot\eta_v(t) = \dot x_{t'} + O(|v^\perp|), \qquad v \in B, \quad t \in (-R,R).
        \]
        Since $\eta_v(t') = L_{t'}v$, by integrating the above estimate from $t'$ to $t$ we get
        \[
            \eta_v(t) = L_{t'}v + (t-t')\cdot \dot x_{t'} + O\left(|t-t'||v^\perp|\right).
        \]
        Applying the isometry $L_{t'}^{-1}$ to both sides and using the definition of $\eta_v$ and $G_t$ we get the desired result.
    \end{proof}

    Define a quadratic form $q$ on $\RR^n$ by
    \[
        q(v) : = \frac12\cdot\Hess r\vert_{x_0}\,(L_0v) = \frac12\cdot\sum_{i=1}^{n-1}\kappa_i(v^i)^2, \qquad v \in \RR^n.
    \]

    The fact that the signed distance function $r$ is differentiable twice in the sense of Alexandrov at $x_0$ gives us the following expansion of $r\circ G_t$:

    \begin{lemma}\label{taylorlemma}
        For  $v \in B$,
        \[
            r(G_t(v)) = t + v^n + q(v) + O(t|v^\perp|) + o(|v|^2+t^2) \qquad \text{ as } (v,t) \to (0,0).
        \]
    \end{lemma}

    \begin{proof}
        Setting $t' =0$ in \eqref{translationeq} and applying $G_0$ to both sides, we see that
        \begin{equation}\label{eq:GtG0}
            d\left(G_t(v),G_0(v + t\cdot e_n)\right) = O(t|v^\perp|).
        \end{equation}
        Since $r$ is 1-Lipschitz, it follows from \eqref{eq:GtG0} and \eqref{alexandroveq} that
        \begin{align*}
            r(G_t(v)) & = r(G_0(v + t\cdot e_n)) + O(t|v^\perp|)\\
            & = r\left(\exp_{x_0}(L_0v + t\cdot\nu\vert_{x_0})\right) + O(t|v^\perp|)\\
            & = \left<L_0v + t\cdot\nu\vert_{x_0},\nu\vert_{x_0}\right> + \frac12\cdot \Hess r(L_0v) + O(t|v^\perp|) + o(|v|^2 + t^2)\\
            & = v^n + t + q(v) + O(t|v^\perp|) + o(|v|^2 + t^2),
        \end{align*}
        as desired.
    \end{proof}

    \begin{lemma}\label{halfspacelemma}
        For every $\eps >0$ there exist $\ell_0,\lambda_0>0$ such that
        \begin{equation}\label{halfplaneeq}
            G_t^{-1}(K)\cap B_\ell(0) \subseteq \left\{v \in B_\ell(0) \subseteq \RR^n \,\, \big\vert \,\, v^n + t \le  \eps\cdot\ell^2\right\}\qquad 
            \begin{array}{c}
                \text{for every $0 < \ell < \ell_0$}\\
                \text{and every $0 < t < \lambda_0\cdot \ell$}.
            \end{array}
        \end{equation}
        See Figure \ref{halfspacefig}.
    \end{lemma}

    \begin{proof}
        Let $\eps > 0$. By Lemma \ref{taylorlemma} and positive semidefiniteness of $q$, there exist $C,\ell_0 > 0$ such that for every $v \in B_\ell(0)$, every $\lambda_0 > 0$, every $0 < \ell < \ell_0$ and every $0 < t < \lambda_0\cdot \ell$,
        \[
            v^n + t \le r(G_t(v)) + C\cdot t|v| + (\eps/2)\cdot (|v|^2+t^2) \le  r(G_t(v)) + \left(C\cdot \lambda_0 + (\eps/2)(1 + \lambda_0^2)\right)\cdot \ell^2.
        \]
        Let $v \in B_\ell(0)$ and suppose that $v$ is not a member of the set on the right-hand side of \eqref{halfplaneeq}. Then $\eps\ell^2 < v^n + t$, so by the above inequality,
        \[
            \left(\eps - C\cdot\lambda_0 - (\eps/2)(1 + \lambda_0^2)\right)\cdot \ell ^2 < r(G_t(v)).
        \]
        Thus, if $\lambda_0$ is small enough that $\lambda_0/(1-\lambda_0^2) < \eps/2C$ then $r(G_t(v)) > 0$, whence
        $G_t(v) \notin K$.
    \end{proof}

    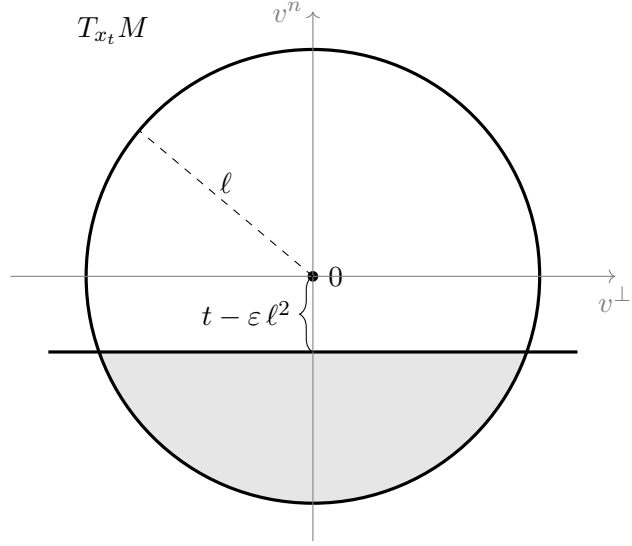
\begin{figure}
        \centering
        \begin{tikzpicture}[scale=1]

            \def\R{3}            
            \def\d{1.0}          

            \node[anchor=west] at (-\R-0.25,\R+0.25) {$T_{x_t}M$};

            \coordinate (O) at (0,0);

            \begin{scope}
                \clip (-\R-0.5,-\R-0.5) rectangle (\R+0.5,-\d);
                \fill[gray!20] (O) circle (\R);
            \end{scope}

            \draw[very thick] (O) circle (\R);

            \draw[very thick] (-\R-0.5,-\d) -- (\R+0.5,-\d);

            \fill (O) circle (2pt)
            node[right = 2pt] {$0$};

            \draw[->,gray] (-\R-1,0) -- (\R+1,0) node[below] {$v^\perp$};
            \draw[->,gray] (0,-\R-.5) -- (0,\R+.5) node[left] {$v^n$};

            \draw[dashed]
            (O) -- ($ (O) + ({140}:\R) $)
            node[midway,above] {$\ell$};

            \draw[decorate,decoration={brace,amplitude=5pt}]
            (0,-\d) -- (0,0)
            node[midway,xshift=-25pt] {$t-\varepsilon\,\ell^2$};
        \end{tikzpicture}
        \caption{Lemma \ref{halfspacelemma} asserts that only points lying in the shaded region can be mapped to $K$ under $G_t$, provided that $\ell < \ell_0$ and $t < \lambda_0 \cdot \ell$.}
        \label{halfspacefig}
    \end{figure}

    For a set $S$ in a vector space we denote by $\CC(0,S)$ the union of all line segments joining the origin to the set $S$ whose relative interiors do not intersect the set $S$.

    \begin{lemma}\label{eucconelemma}
        Let $x \in {U} \setminus K$ and let $V$ be a strongly-convex neighborhood of $x$ contained in ${U}$. Then
        \[
            \exp_x^{-1}\left(\CC(x,K)\cap V\right) \supseteq \CC(0,\exp_x^{-1}(K\cap V)),
        \]
        and if $\CC(x,K) \subseteq V$ then equality holds.
    \end{lemma}
    \begin{proof}
        Since $\CC(x,K)$ consists of minimizing geodesics joining the point $x$ to the set $K$, and the set $V$ is strongly convex, the set $\exp_x^{-1}(\CC(x,K)\cap V)$ contains all line segments joining the origin in $T_xM$ to the set $\exp_x^{-1}(K\cap V)$. The desired inclusion then follows from the definitions of $\CC(x,K)$ and $\CC(0,\exp_x^{-1}(K\cap V))$. If $\CC(x,K) \subseteq V$ then every minimizing geodesic joining $x$ to $K$ is contained in $V$, so its inverse image under $\exp_x$ is contained in $\CC(0,\exp_x^{-1}(K\cap V))$.
    \end{proof}

    \begin{lemma}\label{degeneratelemma}
        Proposition \ref{volumeprop} holds true when $\kappa(x_0) = 0$, i.e. in this case
        \[
            t^{-1}\cdot\Vol(\CC(x_t,K))^{\frac{2}{n+1}}\xrightarrow{t \to 0}\infty.
        \]
    \end{lemma}

    \begin{proof} 
        Let 
        \[
            m \ge 1.
        \]
        Without loss of generality, the vector $w_{n-1} = L_0e_{n-1} \in T_{x_0}M$ is in the kernel of $\sff\vert_{x_0}$, i.e.
        \begin{equation}\label{wnminusoneeq}
                2\cdot q(e_{n-1}) = \sff(w_{n-1}) = 0.
        \end{equation}
        
        By Lemma \ref{taylorlemma}, there exist $C\ge 1$ and $t_0 = t_0(x_0) \le \tfrac{1}{512mC^2}$ such that for all $v \in B$,
        \begin{equation}\label{eq:rGtcylinder}
            r(G_t(v)) \le t + v^n + q(v) + C \cdot t|v| + \tfrac{1}{8m}|v|^2, \qquad t \in (0,t_0).
        \end{equation}
        Write $\hat \kappa:=\max_i\kappa_i$. For $t \in (0,t_0)$, let $Q_t\subseteq \RR^n$ denote the $(n-1)$-dimensional cylinder given by
        \[
            Q_t : = \left\{-\frac32\cdot t\cdot e_n + s \cdot e_{n-1} +  u \quad \middle| \quad -\sqrt{mt} \le s \le \sqrt{mt}, \, \,  u \in \{e_{n-1},e_n\}^\perp, \, \, |u| < \sqrt{at}\,\,\right\},
        \]
        where
        \[
            a : = \min\left\{\,1\,,\,\frac{1}{16\hat \kappa}\right\}
        \]
        if $\hat\kappa > 0$, and $a:=1$ if $\hat\kappa=0$.
        Then 
        \begin{equation}\label{eq:QtinBall}
            Q_t\subseteq B_{2\sqrt{mt}}(0),
        \end{equation}
        and for $v \in Q_t$, the component of $v$ orthogonal to $\operatorname{span}\{e_{n-1},e_n\}$ has norm at most $\sqrt{at}$, whence $q(v)\le \tfrac12\hat\kappa at$ by \eqref{wnminusoneeq}. Since $m \ge 1$, $a \le 1$ and $t < t_0$,
        \[
            |v|^2 = \tfrac94 t^2 + s^2 + |u|^2 \le \tfrac94 t^2 + mt + at \le \left(2m + \tfrac{9t_0}{4}\right)t.
        \]
        Thus, by \eqref{eq:rGtcylinder} and our choices of $a$ and $t_0$, for every $v \in Q_t$ and every $t < t_0$,
        \begin{align*}
            r(G_t(v)) & \le t - \frac32\cdot t + \frac12\hat \kappa \cdot at + C \cdot t|v| + \tfrac{1}{8m}|v|^2\\
            & \le t \cdot \left(-1/2 + \frac12a\hat\kappa + 2C \cdot \sqrt{t_0m} + 1/4 + \frac{9t_0}{32m}\right)
            \\
            & \le -t/8\\
            & < 0,
        \end{align*}
        whence
        \begin{equation}\label{eq:GtQ}
            G_t(Q_t) \subseteq K \qquad \text{ for all $0 <t < t_0$}.
        \end{equation}
        Let
        \[
            t_1:= \min\left\{t_0, \lambda_0^2m,\frac{\ell_0^2}{4m}\right\},
        \]
        where $\ell_0$ and $\lambda_0$ are as in the conclusion of Lemma \ref{halfspacelemma} when we take $\eps = 1/8m$. Let $0 < t < t_1$ and let
        \[
            \ell : = 2\sqrt{mt}.
        \]
        Then 
        \[
            \ell < \ell_0 \qquad \text{ and } \qquad \lambda_0 \cdot \ell > \sqrt{\frac{t}{m}}\cdot 2\sqrt{mt} > t,
        \]
        whence by Lemma \ref{halfspacelemma},
        \begin{equation}
        \begin{split}\label{eq:Gtminus1K}
            G_t^{-1}(K)\cap B_\ell(0) & \subseteq \left\{v \in B_\ell(0) \subseteq \RR^n \,\, \big\vert \,\, v^n + t \le  \frac{\ell^2}{8m}\right\}\\
            & = \left\{v \in B_\ell(0) \subseteq \RR^n \,\, \big\vert \,\, v^n + t/2 \le 0\right\}.
        \end{split}
        \end{equation}

        \medskip
        Recall that $\CC(0,Q_t)$ is the union of all line segments joining the origin to the set $Q_t$ whose relative interiors lie outside the set $Q_t$. It follows from \eqref{eq:QtinBall} that $\CC(0,Q_t) \subseteq B_\ell(0)$. Let 
        \[
            Z_t : = \{v \in \CC(0,Q_t) \,\,\colon\,\, v^n + t/2 > 0\}.
        \]
        If $v \in Z_t$ then $v \notin G_t^{-1}(K)$ by \eqref{eq:Gtminus1K}. But $v \in \CC(0,Q_t)$, and $Q_t \subseteq G_t^{-1}(K)$ by \eqref{eq:GtQ}, so $v$ lies on a line segment joining the origin to the set $G_t^{-1}(K)$, i.e. $v \in \CC(0,G_t^{-1}(K))$. Hence, by Lemma \ref{eucconelemma} and the fact that $\CC(0,\cdot)$ commutes with linear maps,
        \begin{align*}
            Z_t & \subseteq \CC(0,G_t^{-1}(K)) \\
            & = \CC(0,L_t^{-1}(\exp_{x_t}^{-1}(K))) \\
            & = L_t^{-1}\left(\CC(0,\exp_{x_t}^{-1}(K))\right)\\
            & \subseteq L_t^{-1}\left(\exp_{x_t}^{-1}\left(\CC(x_t,K)\right)\right)\\
            & = G_t^{-1}\left(\CC(x_t,K)\right).
        \end{align*}
        Here $G_t^{-1}(K)$ and $\exp_{x_t}^{-1}(K)$ are understood as $G_t^{-1}(K \cap G_t(B'))$ and $\exp_{x_t}^{-1}(K \cap B_{2R}(x_t))$ respectively, so that $G_t$ and $\exp_{x_t}$ restrict to diffeomorphisms. It therefore follows from Lemma \ref{Gtlemma} that
        \begin{align*}
            \Vol(\CC(x_t,K)) & \ge \Vol_{\RR^n}\left(Z_t\right) + O(\ell^{n+2})\\
            & = \Vol_{\RR^n}\left(Z_t\right) + O\left(m^{\frac{n+2}{2}}t^{\frac{n+2}{2}}\right).
        \end{align*}

        The set $\CC(0,Q_t)$ is a cone with base $Q_t$, height $3t/2$ and apex at the origin. Hence, the set $Z_t$ is a cone with base of volume $(1/3)^{n-1}\cdot \Vol(Q_t) = (1/3)^{n-1} \cdot 2\cdot (mt)^{\frac12}\cdot \vartheta_{n-2}\cdot (at)^{\frac{n-2}{2}}$ and height $t/2$, so 
        \begin{align*}
            \Vol(\CC(x_t,K)) & \ge \hat c\cdot m^{\frac12}\cdot t^{\frac{n+1}{2}}\cdot a^{\frac{n-2}{2}} - \hat C\cdot m^{\frac{n+2}{2}}\cdot t^{\frac{n+2}{2}}\\
            & \ge (\hat c/2) \cdot a^{\frac{n-2}{2}}\cdot m^{\frac12}\cdot t^{\frac{n+1}{2}},
        \end{align*}
        where the last inequality holds true provided that
        \[
            t < t_2 : = \min\left\{t_1 \,,\, \frac{\hat c^2}{4\hat C^2}\cdot m^{-n-1} \cdot a^{n-2}\right\}.
        \]
        It follows that 
        \[
            t^{-1}\cdot \Vol(\CC(x_t,K))^{\frac{2}{n+1}} \ge  c'\cdot a^{\frac{n-2}{n+1}}\cdot m^{\frac{1}{n+1}} \qquad \text{ for all $0 < t < t_2$.}
        \]
        Since $m$ is arbitrary and $a$ depends only on $x_0$, this finishes the proof of the lemma.
    \end{proof}

    For the rest of the present section, we assume that
    \[
        \kappa(x_0) > 0.
    \]
    Since the boundary of $K$ is strictly convex at $x_0$, for small $t$ the set $\CC(x_t,K)$ resembles a cone of height $t$ over a paraboloid, and therefore shrinks at rate $\sqrt t$ as $t \to 0$. This is the content of the next lemma. For $0 < t < R$, let $\ell(t)$ denote the smallest positive number such that
    \[
        \CC(x_t,K) \subseteq B_{\ell(t)/2}(x_t),
    \]
    i.e.
    \[
        \ell(t) : = 2 \cdot \sup\{d(x_t, y) : y \in \CC(x_t,K)\},
    \]
    see Figure \ref{smallballfig}.

    \medskip
    In what follows, some of the constants implied by big-O notation are allowed to depend on the point $x_0$. When this is the case, we will use the notation $O_{x_0}(\,\cdot\,)$.
    
    \begin{lemma}\label{smallballlemma} $\ell(t) = O_{x_0}\left(t^{1/2} \,\right)$
    \end{lemma}

    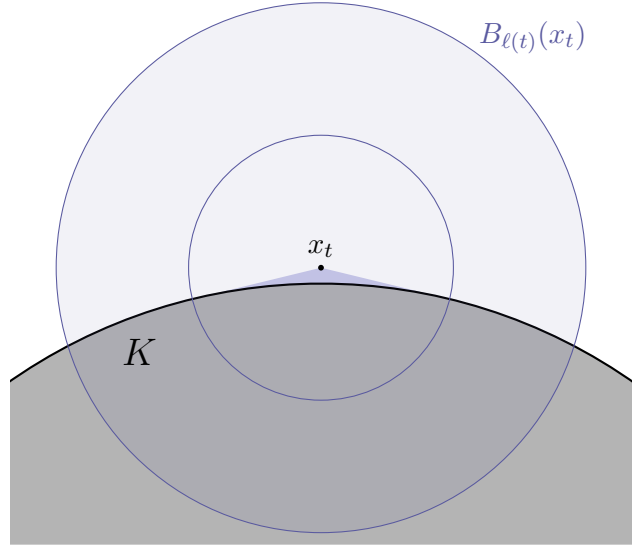
\begin{figure}
        \centering
        \begin{tikzpicture}[scale=.6]
            \definecolor{lightpurple}{RGB}{190,190,230}
            \definecolor{darkpurple}{RGB}{90,90,160}
            \definecolor{darkgray}{RGB}{180,180,180}

            \def\R{12}
            \pgfmathsetmacro{\Cy}{-11.5}       
            \def\py{0.85}                       
            \coordinate (xt) at (0, \py);

            \pgfmathsetmacro{\dd}{\py - \Cy}
            \pgfmathsetmacro{\tx}{\R*sqrt(\dd*\dd - \R*\R)/\dd}
            \pgfmathsetmacro{\ty}{\Cy + \R*\R/\dd}

            \pgfmathsetmacro{\innerR}{sqrt(\dd*\dd - \R*\R)}
            \pgfmathsetmacro{\outerR}{2*\innerR}

            \def\xext{7}
            \pgfmathsetmacro{\arcstart}{180 - acos(\xext/\R)}
            \pgfmathsetmacro{\arcend}{acos(\xext/\R)}

            \pgfmathsetmacro{\clipbot}{\Cy + sqrt(\R*\R - \xext*\xext) - 3.5}
            \pgfmathsetmacro{\cliptop}{\py + \outerR + 0.5}
            \clip ({-\outerR-1}, \clipbot) rectangle ({\outerR+1}, \cliptop);

            \fill[lightpurple, opacity=0.8]
                (xt) -- (-\tx, \ty) -- (\tx, \ty) -- cycle;

            \fill[darkgray]
                ({-\xext}, {\Cy + sqrt(\R*\R - \xext*\xext)})
                arc (\arcstart:\arcend:\R)
                -- (\xext, \clipbot) -- (-\xext, \clipbot) -- cycle;
            \draw[thick]
                ({-\xext}, {\Cy + sqrt(\R*\R - \xext*\xext)})
                arc (\arcstart:\arcend:\R);

            \draw[darkpurple] (xt) circle (\innerR);   
            \draw[darkpurple] (xt) circle (\outerR);   
            \fill[darkpurple, opacity=0.08] (xt) circle (\outerR);   

            \node[darkpurple] at ({\outerR*0.8}, {\py+\outerR*0.87}) {$B_{\ell(t)}(x_t)$};

            \fill (xt) circle (1.8pt);
            \node[above] at (xt) {$x_t$};

            \node at (-4, -1) {\Large $K$};
        \end{tikzpicture}
        \caption{The ball $B_{\ell(t)}(x_t)$ has twice the radius of the smallest ball containing $\CC(x_t,K)$.}
        \label{smallballfig}
    \end{figure}

    \begin{proof}
        If this is not the case, then there exists a sequence $t_j \searrow 0$ and a sequence $y_j \in \CC(x_{t_j},K)$ such that
        \begin{equation}\label{ynfareq}
            \frac{d(x_{t_j},y_j)}{\sqrt {t_j}}\to \infty.
        \end{equation} 
        By the definition of $\CC(x_{t_j},K)$, there exists a sequence $\gamma_j:[0,s_j] \,{\to}\, M$ of unit-speed minimizing geodesics satisfying
        \[
            \gamma_j(0) = x_{t_j}, \quad \gamma_j(s_j) =: z_j \in \partial K, \quad \gamma_j(s)\notin K \quad \forall s \,{\in}\, [0,s_j),
        \]
        and
        \[
            \gamma_j(d(x_{t_j},y_j)) = y_j .
        \]
        Since each $\gamma_j$ is a unit-speed minimizing geodesic, we have $s_j=d(x_{t_j},z_j)\ge d(x_{t_j},y_j)$, so by \eqref{ynfareq},
        \begin{equation}\label{sjoversqrttjeq}
            \frac{s_j}{\sqrt{t_j}} = : M_j \to \infty.
        \end{equation}
        By compactness of $K$, we may assume that $z_j \to z \in \partial K$. 
        
        \medskip
        Suppose that $z \ne x_0$. Upon passing to a subsequence, the geodesics $\gamma_j$ converge to a nonconstant geodesic $\gamma$ joining $x_0$ to $z$. Since $x_0$ and $z$ belong to $K$, assumption \eqref{eq:Kcondition} implies that the geodesic $\gamma$ must be contained in $K$. But the relative interiors of the geodesics $\gamma_j$ are disjoint from $K$, and so $\gamma$ is in fact contained in $\partial K$. This forces $\sff(\dot\gamma(0)) = 0$, a contradiction to the assumption $\kappa(x_0) > 0$.

        \medskip
        We may therefore assume that 
        $z_j \to x_0$,  in which case the vector
        \[
            v_j : = G_{t_j}^{-1}(z_j) \in B
        \]
        is well-defined (for large enough $j$, and without loss of generality for all $j$). We may further assume that the unit vectors
        \[
            u_j : = \frac{v_j}{|v_j|} = \frac{v_j}{s_j}
        \]
        converge to a unit vector $u$. Note that by the definition of $G_{t_j}$,
        \begin{equation}\label{gammajseq}
            \gamma_j(s) = G_{t_j}(su_j), \qquad s \in [0,s_j].
        \end{equation}
        Since $r(z_j)=0$, Lemma \ref{taylorlemma} gives
        \[
            0=\frac{t_j}{s_j}+u_j^n+s_jq(u_j)+O(t_j)+o(s_j).
        \]
        Since $s_j\to0$ and $t_j/s_j=\sqrt{t_j}/M_j\to0$, it follows that $u^n=0$. Hence $u\in e_n^\perp$, and the assumption $\kappa(x_0)>0$ implies that $q(u)>0$.
        By Lemma \ref{taylorlemma},   
        \begin{align*}
            0 = r(z_j)
            & = t_j + v_j^n + q(v_j) + O(t_js_j) + o(s_j^2)\\
            & = t_j + v_j^n + s_j^2q(u_j) + O(t_js_j) + o(s_j^2)\\
            & = t_j\left( 1 +\frac{v_j^n}{t_j} + M_j^2(q(u)+o(1)) + O(s_j)\right).
        \end{align*}
        The expression in parentheses therefore vanishes. But $s_j\to 0$ since $z_j \to x_0$, whence necessarily
        \begin{equation}\label{vnoverMteq}
            \frac{v_j^n}{M_j^2t_j}\to -q(u)<0.
        \end{equation}
        By Lemma \ref{taylorlemma} and by \eqref{sjoversqrttjeq}, \eqref{gammajseq} and \eqref{vnoverMteq},
        \begin{align*}
            r\left(\gamma_j\left(\sqrt{t_j}\right)\right) & = 
            r\left(G_{t_j}\left(\sqrt{t_j}u_j\right)\right)\\
            & = t_j + \sqrt{t_j}\cdot u_j^n + t_j\cdot q(u_j) + O\left(t_j^{3/2}\right) + o(t_j)\\
            & = \frac{1}{M_j}\cdot v_j^n + t_j\cdot (1 + q(u) + o(1)) + o(t_j),
        \end{align*}
        which by \eqref{vnoverMteq} is negative for large $j$, contradicting the fact that $\gamma_j(s) \notin K$ for all $ 0 \le s < s_j$.
    \end{proof}

    \begin{figure}
        \centering
        \begin{tikzpicture}[scale=1]
            \def\R{3}   
            \def\s{1}   

            \draw[->,gray] (-\R-0.5,0) -- (\R+0.5,0) node[below] {$v^\perp$};
            \draw[->,gray] (0,-\R-0.5) -- (0,\R+0.5) node[left] {$v^n$};

            \begin{scope}
                \clip (0,0) circle (\R);
                \draw[very thick,domain=-\R-1:\R+1,samples=400]
                plot (\x, {-0.1*(\x)*(\x) - \s});
            \end{scope}

            \begin{scope}
                \clip (0,0) circle (\R);
                \clip
                plot[domain=-\R-1:\R+1,samples=400] (\x, {-0.1*(\x)*(\x) - \s})
                -- (\R+1,-\R-1) -- (-\R-1,-\R-1) -- cycle;
                \fill[gray!20, opacity = .5] (-\R-1,-\R-1) rectangle (\R+1,\R+1);
            \end{scope}

            \draw[thick, gray] (0,0) circle (\R);
            \node[gray] at (.55*\R,.55*\R) {$B_{\ell(t)}(0)$};

            \node at (1.5,-2) {$P_s$};

            \draw[decorate,decoration={brace,amplitude=5pt}]
            (0,-\s) -- (0,0);

            \node at ({-\s/2},{-\s/2}) {$-s$};
        \end{tikzpicture}
        \caption{The set $P_s$.}
        \label{Ptfig}
    \end{figure}
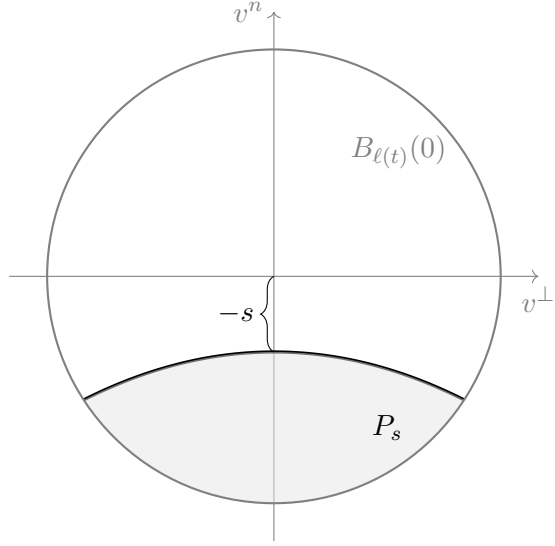

    For $s \in \RR$, write
    \[
        P_s{} : = \left\{\,\,v \in \RR^n\quad \big\vert \quad s + v^n + q(v) \le 0 \right\},
    \]
    see Figure \ref{Ptfig}. It is not hard to see that the cone $\CC(0,P_s)$ is contained in the Euclidean ball $B_{\mu(s)}(0)$, where $\mu(s) := \sqrt{2s/\min_i\kappa_i + 4s^2}$. The next lemma asserts that we can approximate the set $G_t^{-1}(\CC(x_t,K))$ by the set $\CC(0,P_t)$.
    \begin{lemma}\label{sandwichlemma} 
        For every $\eps > 0$ there exists $t_0 > 0$ such that for every $0 < t < t_0$,
        \begin{equation}\label{eq:sandwich}
            \CC\left(0,P_{(1+\eps)t}\right) \setminus P_{(1-\eps)t}
            \subseteq 
            G_t^{-1}\left(\CC(x_t,K)\right)
            \subseteq
            \CC\left(0,P_{(1-\eps)t}\right) \cup \left(P_{(1-\eps)t}\setminus P_{(1+2\eps)t}\right),
        \end{equation}
    \end{lemma}

    \begin{proof}
        By Lemma~\ref{smallballlemma}, there exists a constant $C_{x_0} > 0$ such that $\ell(t) \le C_{x_0}\sqrt t$ for all sufficiently small $t > 0$. Write
        \[
            \mathcal{B}_t : = B_{C_{x_0}\sqrt t}(0) \subseteq \RR^n.
        \]
        By the definition of $\ell(t)$ and $G_t$,
        \begin{equation}\label{CxtKinGtbteq}
           \CC(x_t,K) \subseteq B_{\ell(t)}(x_t) \subseteq G_t(\mathcal{B}_t) = B_{C_{x_0}\sqrt t}(x_t), \qquad 0 < t < R.
        \end{equation}
        Hence, by Lemma \ref{eucconelemma},
        \begin{align*}
            \exp_{x_t}^{-1}(\CC(x_t,K)) & = \CC(0,\exp_{x_t}^{-1}(K \cap B_{\ell(t)}(x_t)))\\
            & = \CC(0,\exp_{x_t}^{-1}(K) \cap \mathcal{B}_t).
        \end{align*}
        Since $G_t = \exp_{x_t}\circ \, L_t$ and the operation $\CC(0,\cdot)$ commutes with linear maps, it follows that
        \[
            G_t^{-1}(\CC(x_t,K)) = \CC(0,G_t^{-1}(K) \cap \mathcal{B}_t).
        \]
        By the definition of $\CC(0,\,\cdot\,)$, this means that
        \begin{equation}\label{inconeq}
            v \in G_t^{-1}\left(\CC(x_t,K)\right)\setminus\{0\}
            \quad
            \iff
            \quad
            \begin{array}{c}
                \forall \,\,0 < s < 1 :\,\, r(G_t(sv)) > 0 \quad \text{ and }\vspace{.1cm}\\
                \exists \,\, s_0 \ge 1 :\, \, r(G_t(s_0v)) \le 0.
            \end{array}
        \end{equation}

        \medskip
        Fix $\eps > 0$. By Lemma \ref{taylorlemma} and Lemma \ref{smallballlemma}, there exists $t_0 = t_0(x_0,\eps) >0$ such that if $0<t<t_0$ then for every $v \in \mathcal{B}_t$,
        \begin{equation}\label{rGtsvtayeq}
            (1-\eps)t + sv^n + s^2\cdot q(v) \le r(G_t(sv)) \le (1+\eps)t + sv^n + s^2\cdot q(v)\qquad \text{ for all $0 < s < \frac{{C_{x_0}\sqrt t}}{|v|}$}.
        \end{equation}

        Write
        \[
            f(v) : = v^n + q(v), \qquad v \in \RR^n.
        \]
        By \eqref{inconeq} and \eqref{rGtsvtayeq},

        \begin{align}
            G_t^{-1}(\CC(x_t,K)) \setminus\{0\}
            &\subseteq
            \left\{
            v \in \mathcal{B}_t \quad \middle| \quad 
            \begin{array}{c}
                \forall \, \, 0 < s < 1: f(sv) > -(1+\eps)t,\\
                \exists s_0 \ge 1: f(s_0v) \le -(1-\eps)t
            \end{array}\nonumber
            \right\}\\
            &\subseteq
            \left\{
            v \in \mathcal{B}_t \quad \middle| \quad 
            \begin{array}{c}
                f(v) \ge -(1+\eps)t,\\
                \exists s_0 \ge 1: f(s_0v) \le -(1-\eps)t
            \end{array}
            \right\}\label{Gtsubseteq}
        \end{align}

        and 

        \begin{align}
            G_t^{-1}(\CC(x_t,K)) 
            &\supseteq
            \left\{
            v \in \mathcal{B}_t \quad \middle| \quad 
            \begin{array}{c}
                \forall \, \, 0 < s < 1: f(sv) > -(1-\eps)t,\\
                \exists s_0 \ge 1: f(s_0v) \le -(1+\eps)t
            \end{array}
            \right\}\nonumber
            \\
            &\supseteq
            \left\{
            v \in \mathcal{B}_t \quad \middle| \quad 
            \begin{array}{c}
                f(v) \ge -(1-\eps)t,\\
                \exists s_0 \ge 1: f(s_0v) \le -(1+\eps)t
            \end{array}
            \right\},\label{Gtsupseteq}
        \end{align}
        where the second inclusion holds since the function $s \mapsto f(sv)$ is convex and vanishes at $0$. Indeed, for $v$ in the lower set we have $f(s_0 v) \le -(1+\eps)t < -(1-\eps)t \le f(v)$ for some $s_0 \ge 1$ whence $s \mapsto f(sv)$ is strictly decreasing on $[0,1]$ and in particular $f(sv) > f(v) \ge -(1-\eps)t$ for all $0 < s < 1$.

        \medskip
        To prove the first inclusion in \eqref{eq:sandwich}, we decrease $t_0$ if necessary so that $\CC(0,P_{(1+\eps)t}) \subseteq \mathcal{B}_t$ for all $0<t<t_0$. Let $v \in \CC(0,P_{(1+\eps)t}) \setminus P_{(1-\eps)t}$. By the definition of $\CC(0,P_{(1+\eps)t})$, there exists $s_0 \ge 1$ such that $f(s_0 v) \le - (1+\eps)t$, while $f(v) > -(1-\eps)t$ since $v \notin  P_{(1-\eps)t}$. Hence by \eqref{Gtsupseteq}, $v \in G_t^{-1}(\CC(x_t,K))$. This proves the first inclusion.

        \medskip
        For the second inclusion in \eqref{eq:sandwich}, let $v \in G_t^{-1}(\CC(x_t,K))$; since $0 \in \CC\left(0,P_{(1-\eps)t}\right)$, we may assume that $v \ne 0$. By the definition of $\ell(t)$, 
        \[
            v \in B_{\ell(t)/2} \subseteq \mathcal{B}_t.
        \]
        In order to prove that $v \in \CC\left(0,P_{(1-\eps)t}\right) \cup \left(P_{(1-\eps)t}\setminus P_{(1+2\eps)t}\right)$, it suffices to show that
        \[ 
            \text{if $v \notin P_{(1-\eps)t}\setminus P_{(1+2\eps)t}$ then $v \in \CC(0,P_{(1-\eps)t})$.}
        \]
        The assumption $v \notin P_{(1-\eps)t}\setminus P_{(1+2\eps)t}$ means that either $f(v) > -(1-\eps)t$ or $f(v) \le -(1+2\eps)t$, but the latter is impossible by \eqref{Gtsubseteq} since $v \in G_t^{-1}(\CC(x_t,K))$. Thus $f(v) > -(1-\eps)t$ and therefore $v \notin P_{(1-\eps)t}$, while by \eqref{Gtsubseteq} there exists $s_0 > 1$ such that $f(s_0v) \le -(1-\eps)t$, i.e. $s_0v \in P_{(1-\eps)t}$. It follows that $v \in \CC(0,P_{(1-\eps)t})$.
    \end{proof}

    We can now prove Proposition \ref{volumeprop} in the case $\kappa(x_0)  > 0$:
    \begin{lemma}\label{nondegeneratelemma}
        Suppose that $\kappa(x_0) > 0$. Then
        \begin{equation}\label{VolGtminusoneeq}
            \Vol(\CC(x_t,K)) = \left(\frac{t}{\beta_n\cdot \kappa(x_0)^{\frac{1}{n+1}}}\right)^{\frac{n+1}{2}}(1 + o(1)) \qquad \text{as $t\searrow 0$}.
        \end{equation}
    \end{lemma}

    \begin{proof}
        Recall that $P_s : = \{v \in \RR^n : s + v^n + q(v) \le 0\}$, where $q(v) = \tfrac12 \sum_{i=1}^{n-1}\kappa_i(v^i)^2$, and that $\kappa(x_0) = \kappa_1\cdots\kappa_{n-1}$.
        It is an elementary calculus exercise to verify that for every $s \in (0,R)$,
        \[
            \CC(0,P_s)\subseteq B_{\mu(s)}(0), \quad \text{where} \quad \mu(s): = \sqrt{\frac{2s}{\min_i\kappa_i} + 4s^2},
        \]
        and that the Euclidean volume of $\CC(0,P_s)$ is given by
        \begin{equation}\label{VolC0Pseq}
            \Vol_{\RR^n}\left(\CC(0,P_s)\right) = \frac{2^{\frac{n+1}{2}}\vartheta_{n-1}s^{\frac{n+1}{2}}}{n(n+1)\sqrt{\kappa(x_0)}} = \left(\frac{s}{\beta_n\cdot \kappa(x_0)^{\frac{1}{n+1}}}\right)^{\frac{n+1}{2}}.
        \end{equation}
        Let $\eps > 0$. Using the fact that $\ell(t) = O_{x_0}(t^{1/2})$ (Lemma \ref{smallballlemma}), it is also not hard to check that for every $0 < t < R$,
        \[
            \Vol_{\RR^n}\!\left((P_{(1-\eps)t}\setminus P_{(1+\eps)t})\cap B_{C_{x_0}\sqrt t}(0)\right) = O_{x_0}\!\left(\eps t\cdot t^{\frac{n-1}{2}}\right) = O_{x_0}(\eps t^{\frac{n+1}{2}}).
        \]
        Since $\CC(0,P_{(1+\eps)t})\cap P_{(1+2\eps)t} = \varnothing$, we deduce from Lemma \ref{sandwichlemma} that there exists $t_0 = t_0(x,\eps)$ such that the set $G_t^{-1}(\CC(x_t,K))$ contains the set $\CC(0,P_{(1+\eps)t})$ up to a set of measure $O_{x_0}(\eps t^{\frac{n+1}{2}})$, and is contained in the set $\CC(0,P_{(1-\eps)t})$ up to a set of measure $O_{x_0}(\eps t^{\frac{n+1}{2}})$. Thus \eqref{VolGtminusoneeq} follows from \eqref{VolC0Pseq} and Lemma \ref{Gtlemma}.
    \end{proof}
    \subsection{Step II: uniform lower bounds}

    In this section we obtain lower bounds on the volume of $\CC(x,K)$ for all $x \in U\setminus K$. This will allow us to apply the dominated convergence theorem in the next step. We begin by describing a simple construction. 
    \begin{definition}[Right circular cone]\label{circconedef}\normalfont
        Let $x \in M$ and let $b,h > 0$. A set $Y \subseteq M$ is called a \emph{right circular cone with apex at $x$, height $h$ and radius $b$} if there exists $y \in M$ with $d(x,y) = h$ such that $Y$ is the union of all minimizing geodesics joining the point $x$ to a point of the form $z = \exp_y(v)$, where $v$ is a tangent vector at $y$ which is perpendicular to the geodesic joining $x$ to $y$ and which has norm $\le b$. See Figure \ref{conefig}. The set of such points $z$ will be called the \emph{base} of the cone.
    \end{definition}

    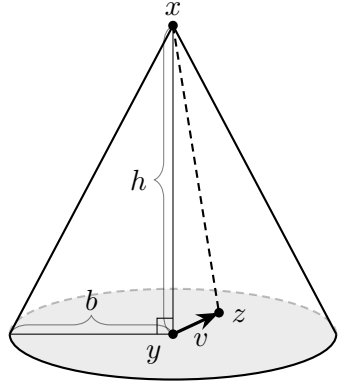
\begin{figure}
        \centering
        \begin{tikzpicture}[scale=1.2,>=Stealth]
            \def\Rx{1.8}      
            \def\Ry{0.5}      
            \def\H{3.4}       
            \def\theta{25}    

            \coordinate (Y) at (0,0);      
            \coordinate (X) at (0,\H);     
            \coordinate (R) at (\Rx,0);
            \coordinate (L) at (-\Rx,0);

            \fill[gray!15] (0,0) ellipse [x radius=\Rx, y radius=\Ry];

            \draw[densely dashed, thick, opacity = .25] (\Rx,0) arc[start angle=0,end angle=180,
            x radius=\Rx, y radius=\Ry];
            \draw[thick] (-\Rx,0) arc[start angle=180,end angle=360,
            x radius=\Rx, y radius=\Ry];

            \draw[decorate,decoration={brace,amplitude=7pt}, opacity = .5]
            (L) -- (Y);
            \draw[decorate,decoration={brace,amplitude=7pt}, opacity = .5]
            (Y) -- (X);

            \draw[thick] (X) -- (R);
            \draw[thick] (X) -- (L);

            \draw (X) -- (Y) node[midway,left=6pt] {$h$};

            \path[name path=ray] (Y) -- ++(\theta:4);
            \path[name path=ellipse] (0,0) ellipse [x radius=\Rx, y radius=\Ry];
            \path[name intersections={of=ray and ellipse,by=P}];


            \coordinate (V) at ($(Y)!0.55!(P)$);

            \draw[line width=1.2pt,->] (Y) -- (V);

            \node[below=10pt, left] at (V) {$v$};
            \node[below=1pt, right = 1pt] at (V) {$z$};

            \draw[densely dashed, thick] (X) -- (V);

            \path[name path=leftRay] (Y) -- ++(180:4);
            \path[name intersections={of=leftRay and ellipse,by=PL}];
            \draw (Y) -- (PL)
            node[pos=0.5,above=9pt,inner sep=0pt] {$b$};

            \pic[draw,angle radius=6pt]{right angle=X--Y--L};

            \fill (X) circle (1.5pt);
            \fill (Y) circle (1.5pt);
            \fill (V) circle (1.5pt);

            \node[above] at (X) {$x$};
            \node[below left=.5pt] at (Y) {$y$};

        \end{tikzpicture}
        \caption{A right circular cone.}
        \label{conefig}
    \end{figure}

    \begin{lemma}\label{conevolumelemma}
        There exists $c_1 > 0$ such that for every $h,b \in (0,R)$, a right circular cone with height $h$ and radius $b$ which is contained in the set $U$ has volume $\ge c_1\cdot h\cdot b^{n-1}$.
    \end{lemma}

    \begin{proof}
        Let $Y$ be a right circular cone contained in the set $U$, with $x,y,b,h$ as in Definition \ref{circconedef}. We claim that there exist constants $c,c'>0$ such that $Y \supseteq \exp_x(Y_0)$, where 
        \begin{equation}\label{eq:Y0def}
            Y_0 : = \left\{w \in T_xM \quad \middle| \quad \theta(w) \le \arctan(cb/h), \quad |w|\cos(\theta(w)) \le c'h \right\},
        \end{equation}
        and where $\theta(w) \in [0,\pi]$ denotes the angle between $w$ and the geodesic joining $x$ to $y$. Since $Y_0$ is a Euclidean circular cone with apex at the origin, height $c'h$ and radius $c'cb$, its Euclidean volume is $c'' \cdot h \cdot b^{n-1}$ for some constant $c'' > 0$, so Lemma \ref{normallemma} would then imply that $\Vol(Y) \ge c_1\cdot h\cdot b^{n-1}$ as desired, for an appropriate choice of $c_1$.

        \medskip
        Define $Y_0$ by \eqref{eq:Y0def}, let $w \in Y_0$ and let $\theta : = \theta(w)$. Our goal is to show that $\exp_x(w) \in Y$. By the definition of $Y$, this amounts to proving that for some $s \ge 1$, the point $\exp_x(sw)$ lies in the base of the cone $Y$.
        
        \medskip
        Let $z' : = \exp_x(2h\cdot w/|w|)$. By Lemma \ref{lem:smalltriangle}, if the constant $c$ (and therefore the angle $\theta$) is smaller than some constant depending only on the set $K$, then $\angle xyz' > \pi/2$. Thus, by continuity, there exists $s \in (0,2h)$ such that the point $z : = \exp_x(sw/|w|)$ satisfies $\angle xyz = \pi/2$, see Figure \ref{fig:conevolumelemma}. By Lemma \ref{righttrianglelemma}, if $c,c'$ are sufficiently small then
        \[
        h\tan\theta \ge c \cdot d(y,z) \qquad \text{ and } \qquad d(x,z)\cdot \cos\theta \ge c' h.
        \]
        By the definition of $Y_0$,
        \begin{itemize}
            \item[-] $h \tan \theta \le c \cdot b$, and therefore $d(y,z) \le b$, whence the point $z$ belongs to the base of the cone $Y$, and
            \item[-] $|w| \le c'\cdot h / \cos\theta \le d(x,z)$, whence the point $\exp_x(w)$ lies on the geodesic joining $x$ to $z$.
        \end{itemize} 
        It follows that $\exp_x(w) \in Y$.
    \end{proof}

    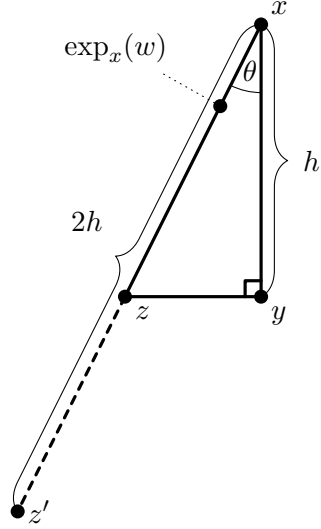
\begin{figure}
        \centering
        \begin{tikzpicture}[scale=1.2, line cap=round, line join=round]
            \useasboundingbox (-2.5,-2) rectangle (2.5,4.8);

            \coordinate (X)  at (1.5, 4);
            \coordinate (Y)  at (1.5, 1);
            \coordinate (Z)  at (0, 1);
            \coordinate (Zp) at (-1.183, -1.367);
            \coordinate (W)  at (1.05, 3.1);

            \draw[very thick] (X)--(Y);      
            \draw[very thick] (Z)--(Y);      
            \draw[very thick] (X)--(Z);      

            \draw[very thick, dashed] (Z)--(Zp);

            \pic[draw, very thick, angle radius=6pt]{right angle=X--Y--Z};

            \fill (X)  circle (2.2pt);
            \fill (Y)  circle (2.2pt);
            \fill (Z)  circle (2.2pt);
            \fill (Zp) circle (2.2pt);
            \fill (W)  circle (2.2pt);

            \node[above right] at (X)  {$x$};
            \node[below right]       at (Y)  {$y$};
            \node[below right]        at (Z)  {$z$};
            \node[right]        at (Zp) {$z'$};
            \node (Wlabel) at (-0.1, 3.75) {$\exp_x(w)$};
            \draw[dotted, thin, shorten <=2.5pt] (W) -- (Wlabel);

            \draw[decorate, decoration={brace, amplitude=10pt, mirror}]
                (X)--(Zp);
            \path (X)--(Zp) node[pos=0.5, above left=10pt] {$2h$};

            \draw[decorate, decoration={brace, amplitude=10pt}]
                (X)--(Y);
            \path (X)--(Y) node[pos=0.5, right=12pt] {$h$};

            \pic[draw, -, "$\theta$", angle eccentricity=0.7, angle radius=.9cm]
                {angle = Z--X--Y};

        \end{tikzpicture}
        \caption{Proof of Lemma \ref{conevolumelemma}.}
        \label{fig:conevolumelemma}
    \end{figure}

    For $x \in \Sigma = \partial K$, let
    \begin{equation}\label{eq:rhodef}
        \rho(x) 
        : =
        \sup
        \left\{
        \rho \in [0,R] 
        \quad \middle| \quad
        \begin{array}{c}
            \text{there exists a closed ball of radius $\rho$ contained}\\
            \text{in $K$ and containing $x$ on its boundary}
        \end{array}
        \right\}.
    \end{equation}
    A singleton is considered a closed ball of radius zero.

    \begin{lemma}\label{rhopositivelemma}
        Let $x \in \Sigma$. If $\rho(x) > 0$ then $x \in \Sigma'$.
    \end{lemma}
    \begin{proof}
        If $\rho(x) > 0$ then the tangent cone $T_xK$ contains the tangent cone of some ball at $x$, which is a half-space. Hence $T_xK$ is a half-space and $x \in \Sigma'$.
    \end{proof}

    \begin{lemma}\label{touchingballlemma}
        Let $x \in \Sigma'$ and let $\rho \in (0,R]$. If $B_\rho(p) \subseteq K$ is a closed ball whose boundary contains $x$, then $p = \gamma_x(-\rho)$.
    \end{lemma}
    \begin{proof}
        Since $B_\rho(p) \subseteq K$ and $x \in \partial B_\rho(p) \cap \partial K$, the outward unit normal to $B_\rho(p)$ at $x$ coincides with the outward unit normal $\nu\vert_x$ to $K$ at $x$, which is well-defined because $x \in \Sigma'$. Therefore $p = \exp_x(-\rho\,\nu\vert_x) = \gamma_x(-\rho)$.
    \end{proof}

    Recall that $\Pi : U \to K$ is the nearest-point projection. We will now show that a lower bound on $\rho(\Pi x)$ can be used to bound $\Vol(\CC(x,K))$ from below by finding a right circular cone contained in $\CC(x,K)$.

    \begin{lemma}\label{coneinclusionlemma}
        There exists $c_2 > 0$ such that for every $x\in {U} \setminus K$ with $r(x) \le \rho(\Pi x)$, the set $\CC(x,K)$ contains a right circular cone with apex at $x$, height $h(x)$ and radius $b(x)$, where
        \begin{equation}\label{hxbxeq}
            h(x) := \tfrac12\cdot r(x) \qquad \text{ and } \qquad b(x) :=  c_2\cdot \sqrt{r(x)\cdot \rho(\Pi x)}.
        \end{equation}
    \end{lemma}
    \begin{proof}
        Let $x \in U \setminus K$ satisfy $r(x) \le \rho(\Pi x)$. Write
        \[
            x_0 : = \Pi x, \qquad t : = r(x) \qquad \text{ and } \qquad \rho : = \rho(\Pi x).
        \]
        Let $Y$ denote a right circular cone with apex at $x$, height $t/2$ and radius $b = c_2\sqrt{t\rho}$, taking the point $y$ in Definition \ref{circconedef} (namely, the center of the base of the cone) to be the midpoint between $x$ and $x_0$. The constant $c_2$ will be chosen later. In order to show that $Y \subseteq \CC(x,K)$, we must prove that:
        \begin{enumerate}[(I)]
            \item every geodesic $\gamma$ joining the point $x$ to the base of $Y$ eventually reaches the set $K$, and
            \item the cone $Y$ is disjoint from the interior of $K$.
        \end{enumerate}

        By Lemma~\ref{rhopositivelemma}, $x_0 \in \Sigma'$. By the definition of $\rho$, the set $K$ contains a closed ball of radius $\rho$ whose boundary contains $x_0$, and by Lemma~\ref{touchingballlemma}, its center is $p := \gamma_{x_0}(-\rho)$. Let $\gamma$ be a unit-speed geodesic joining the point $x$ to a point $z$ lying on the base of the cone, see Figure \ref{fig:twotriangles2}.

        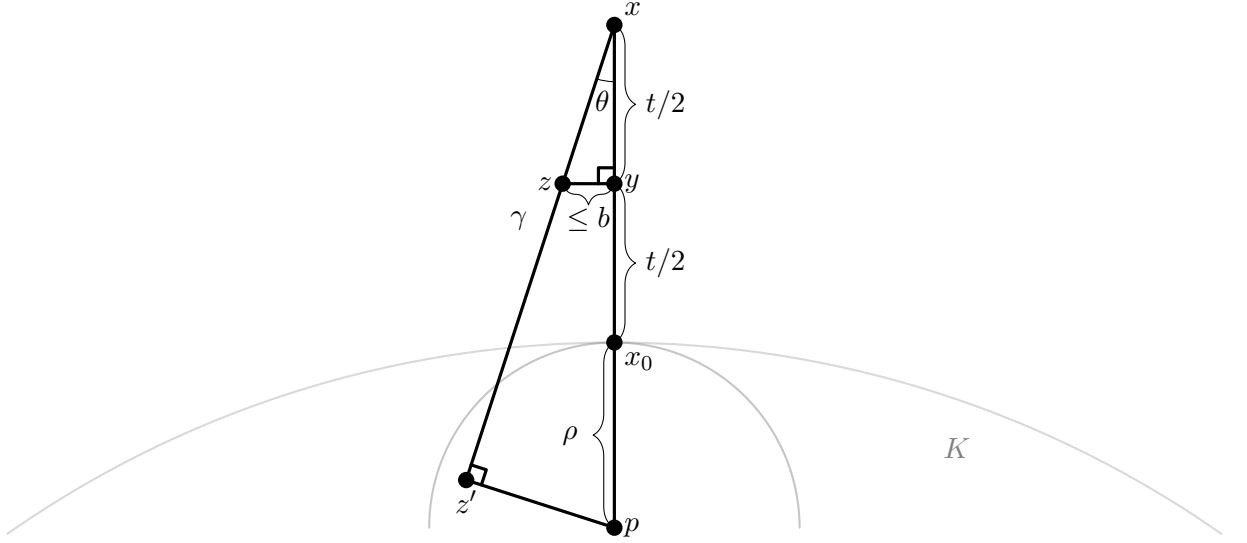
\begin{figure}[tbp]
            \centering
            \begin{tikzpicture}[scale=1.4, line cap=round, line join=round]
                \useasboundingbox (-3.5,-1.05) rectangle (2,4);

                \coordinate (X) at (0,4);
                \coordinate (P) at (0,-0.75); 
                \coordinate (W) at (-1.4,-0.3);
                \coordinate (Xzero) at (0,1);
                \coordinate (Y) at (0,2.5);     

                \path[name path=XW] (X)--(W);
                \path[name path=h0] (-3,2.5)--(3,2.5);
                \path[name intersections={of=XW and h0,by=Z}];

                \draw[gray, line width=0.8pt, opacity=0.45]
                (P) ++(0:1.75) arc[start angle=0, end angle=180, radius=1.75];

                \draw[gray, line width=0.8pt, opacity=0.3]
                (0,-9) ++(55:10) arc[start angle=55, end angle=125, radius=10];
                \node[right = 120pt, gray] at (0,0) {$K$};

                \draw[very thick] (X)--(P);        
                \draw[very thick] (W)--(X);        
                \draw[very thick] (W)--(P);        
                \draw[very thick] (Z)--(Y);        

                \pic[draw,very thick,angle radius=6pt]{right angle=X--W--P};
                \pic[draw,very thick,angle radius=6pt]{right angle=Z--Y--X};

                \fill (X) circle (2.2pt);
                \fill (P) circle (2.2pt);
                \fill (W) circle (2.2pt);
                \fill (Xzero) circle (2.2pt);
                \fill (Y) circle (2.2pt);
                \fill (Z) circle (2.2pt);

                \node[above right] at (X) {$x$};
                \node[right] at (P) {$p$};
                \node[below] at (W) {$z'$};
                \node[below right] at (Xzero) {$x_0$};
                \node[right] at (Y) {$y$};
                \node[left] at (Z) {$z$};

                \path (X) -- (W) node[pos=0.43, left = 5pt] {$\gamma$};
                \path (X) -- (Y) node[pos=0.5, right = 8pt] {$t/2$};
                \path (Y) -- (Xzero) node[pos=0.5, right = 8pt] {$t/2$};
                \path (Xzero) -- (P) node[pos=0.5, left = 10pt] {$\rho$};
                \path (Z) -- (Y) node[pos=0.5, below = 5pt] {$\le b$};

                \draw[decorate,decoration={brace,amplitude=8pt}]
                (X) -- (Y);
                \draw[decorate,decoration={brace,amplitude=8pt}]
                (Y) -- (Xzero);
                \draw[decorate,decoration={brace,amplitude=8pt}]
                (P) -- (Xzero);
                \draw[decorate,decoration={brace,mirror,amplitude=8pt}]
                (Z) -- (Y);

                \pic [draw, -, "$\theta$", angle eccentricity=1.35, angle radius=0.75cm]
                {angle = Z--X--P};
            \end{tikzpicture}
            \caption{Proof of Lemma \ref{coneinclusionlemma}.}
            \label{fig:twotriangles2}
        \end{figure}

        \subsubsection*{Case 1: $\rho^3 \le t \le \rho$.}
        \begin{enumerate}[(I)]
            \item
            Let $z'$ denote the closest point to $p$ on the geodesic $\gamma$, see Figure \ref{fig:twotriangles2}. Since the point $z'$ minimizes distance to $p$, the triangle $\triangle xz'p$ is a right triangle ($\angle xz'p = \pi/2$), whereas $\angle xyz = \pi/2$ by the definition of a right circular cone. By Lemma~\ref{righttrianglelemma} applied to the right triangles $\triangle xyz$ and $\triangle xz'p$, if $c_2$ is small enough then $z$ lies between $x$ and $z'$ on the geodesic $\gamma$. Since the ball $B_{\rho }(p)$ is contained in $K$, in order to prove that the geodesic $\gamma$ eventually reaches the set $K$, it suffices to show that

            \begin{equation}\label{distlerho0eq}
                d(z',p) \le \rho .
            \end{equation}

            \medskip
            By the definition of a right circular cone, 
            \begin{equation}\label{dzx0leneq}
                d(z,y) \le b.
            \end{equation}
            Hence, by Lemma \ref{righttrianglelemma},
            \[
                d(z,x) = \sqrt{t^2/4 + d(y,z)^2} \cdot \left(1 + O\left(t^2 + b^2\right)\right), \qquad d(z,x)\cdot \sin\theta = d(y,z) \cdot (1 + O(t^2+b^2))
            \]
            and
            \[
                (t+\rho )\sin\theta = d(z',p)(1 + O(t^2 + \rho ^2)).
            \]
            These estimates, together with \eqref{dzx0leneq}, give
            \begin{align*}
                d(z',p) & = \frac{(t + \rho )\cdot d(y,z)}{\sqrt{t^2/4 + d(y,z)^2}}\cdot \frac{1 + O(t^2 + b^2)}{1 + O(t^2 + \rho^2)}\\
                & \le \rho \cdot \frac{t/\rho + 1}{\sqrt{t^2/4b^2 + 1}}\cdot \frac{1 + O(t^2 + b^2)}{1 + O(t^2 + \rho^2)}.
            \end{align*}
            Since $t\le \rho$ and $b = c_2\cdot\sqrt{t\rho}$, it follows that
            \begin{equation}\label{eq:dzprimep}
                d(z',p) \le \rho \cdot\frac{t/\rho + 1}{\sqrt{(1/4c_2^2)\cdot t/\rho + 1}}\cdot (1 + C \rho^2).
            \end{equation}
            Since $\rho^3 \le t \le \rho$,
            \begin{align*}
                1 - \left(\frac{t/\rho + 1}{\sqrt{(1/4c_2^2)\cdot t/\rho + 1}}\right)^{2}
                & = \left(\frac{1}{4c_2^2} - 2 - \frac{t}{\rho}\right)\cdot\frac{t/\rho}{(1/4c_2^2)\cdot t/\rho + 1}
                \\
                & \ge \left(\frac{1}{4c_2^2}-3\right)\cdot\frac{\rho^2}{(1/4c_2^2)\cdot \rho^2 + 1}
                \\
                & \ge \frac{\rho^2}{2}\left(\frac{1}{4c_2^2}-3\right)
                \\
                & \ge 2C\rho^2,
            \end{align*}
            provided that $c_2$ and $R$ are sufficiently small. Hence, by \eqref{eq:dzprimep},
            \[
                d(z',p) \le \rho \cdot \sqrt{1 - 2C\rho^2} \cdot (1 + C\rho^2) \le \rho,
            \]
            which gives \eqref{distlerho0eq}.

            \item
            Work in normal coordinates $(x',x^n)$ centered at $x_0$, with the last coordinate corresponding to the direction $\nu\vert_{x_0}$. Since $K$ is convex and $B_R(x_0)$ is strongly convex, in this coordinate system
            \[
                K\cap B_R(x_0)\subseteq \{(x',x^n): x^n\le 0\}.
            \]

            Since $t<R$, $\rho\le R$, and $b=c_2\sqrt{t\rho}$, choosing $c_2$ small enough ensures that the cone $Y$ lies in $B_R(x_0)$. By Lemma \ref{lem:normal-geodesic-estimate}\eqref{normal-geodesic-initial}, the base of $Y$ is contained in $\{x^n\ge t/2-Cb^2\}$, and by Lemma \ref{lem:normal-geodesic-estimate}\eqref{normal-geodesic-chord} the last coordinate of a geodesic $\gamma$ joining $x$ to the base of the cone $Y$ satisfies
            \[
                \gamma^n(s) \ge t/2 - C'(s^2+b^2)\ge s\cdot\left(1/2 - C'R - C'c_2^2\rho\right).
            \]

            Hence, if $c_2$ and $R$ are smaller than some constant depending only on $K$ then $x^n>0$ throughout $Y$, whence $Y$ is disjoint from $K$.
        \end{enumerate}
        \subsubsection*{Case 2: $t \le \rho^3$.}
        \begin{enumerate}[(I)]
            \item
            We prove that in this case, if the constant $c_2$ is small enough then $\gamma(\sqrt{t\rho}) \in K$. To this end, we consider the first order Taylor expansion of $d(\gamma(s),p)$ at $0$ and use it to show that $d(\gamma(\sqrt{t\rho}),p) \le \rho$, which implies that $\gamma(\sqrt{t\rho}) \in K$ since $B_\rho(p) \subseteq K$.

            \medskip
            By the first variation formula, 
            \[
                \left.\frac{d}{ds}\right|_{s=0}d(\gamma(s),p) = -\cos\theta
            \]
            where $\theta = \angle zxy$. By the triangle inequality and the assumption $t \le \rho^3$, for every $s \le \sqrt{t\rho}$,
            \[
                d(\gamma(s),p) \ge d(x,p) - s = t + \rho - s \ge t + \rho - \sqrt{t\rho} \ge c \cdot \rho
            \]
            whence by Lemma \ref{hessianlemma},
            \begin{align*}
                 \frac{d^2}{ds^2}d(\gamma(s),p) & = O\left(\frac{1}{d(p,\gamma(s))}\right) = O\left(\frac{1}{\rho}\right).
            \end{align*}
            By Lemma \ref{righttrianglelemma} and the definition $b = c_2\sqrt{t\rho}$,
            \[
            \begin{aligned}
            \cos\theta 
            &\ge \frac{t/2}{\sqrt{t^2/4 + b^2}}\cdot (1 + O(t^2 + b^2))\\
            &= \frac{t/2}{\sqrt{t^2/4 + c_2^2t\rho}}\cdot (1 + O(t^2 + t\rho))\\
            &= \frac{1}{\sqrt{1 + 4c_2^2\rho/t}}\cdot(1 + O(t^2 + t\rho)).
            \end{aligned}
            \]
            Putting together the above estimates we get
            \begin{align*}
                d\left(\gamma\left(\sqrt{t\rho}\,\right),p\right) & = d(x,p) - \sqrt{t\rho}\cdot \cos \theta + O\left(\frac{1}{\rho}\cdot t\rho\right)\\
                & \le  t + \rho - \frac{\sqrt{t\rho}}{\sqrt{1 + 4c_2^2\rho/t}}\cdot(1 + O(t^2+ \rho t)) + O(t)\\
                & = \rho  + t \cdot \left(1 - \frac{1}{\sqrt{t/\rho + 4c_2^2}} + O(1)\right)\\
                & \le \rho + t\cdot\left(C - 1/(2c_2) + O(t/\rho)\right),
            \end{align*}
            which is again less than $\rho$ if $c_2$ is taken small enough. Thus the geodesic $\gamma$ eventually reaches the set $K$.
            \item
            The same argument as in the previous case shows that the cone $Y$ is disjoint from the set $K$.
        \end{enumerate}
        The proof of the lemma is complete.
    \end{proof}

    \begin{lemma}\label{lem:lowerbound}
        There exists $C_3>0$ such that, for every $x_0 \in \Sigma$ with $\rho(x_0) > 0$,
        \begin{equation}\label{eq:containmentcor}
            \gamma_{x_0}^{-1}(K^\delta\setminus K) \subseteq \left[0,C_3 \cdot \delta^{\frac{2}{n+1}}\cdot \rho(x_0)^{-\frac{n-1}{n+1}}\right]
        \end{equation}
        for all sufficiently small $\delta$.
    \end{lemma}

    \begin{proof}
        Let $x_0 \in \Sigma$ satisfy $\rho(x_0) >0$. By Lemma \ref{rhopositivelemma}, $x_0 \in \Sigma'$. Let $\gamma_{x_0}$ denote the perpendicular to $K$ at $x_0$ and write $x_t := \gamma_{x_0}(t)$, so that
        \[
            r(x_t) = t \qquad \text{ and } \qquad \Pi\left(x_t\right) = x_0
            \qquad \text{for every $t \in (0,R)$.}
        \]

        By Lemma \ref{coneinclusionlemma}, for every $t \in (0,\rho(x_0)]$ the set $\CC(x_t,K)$ contains a right circular cone of height $t/2$ and radius $c_2\cdot\sqrt{t\,\rho(x_0)}$, which by Lemma
        \ref{conevolumelemma} has volume $\ge c\cdot  t^{\frac{n+1}{2}}\cdot \rho(x_0)^{\frac{n-1}{2}}$ where $c=c_1c_2^{n-1}/2$. Hence
        \begin{equation}\label{eq:volumecor}
            \Vol(\CC(x_t,K)) \ge c\cdot t^{\frac{n+1}{2}}\cdot \rho(x_0)^{\frac{n-1}{2}} \qquad \text{for every $t \in \left( 0,\rho(x_0)\right]$.}
        \end{equation}
        
        Let $t \in \gamma_{x_0}^{-1}(K^\delta\setminus K)$. Then
        \begin{equation}\label{eq:VolCCxtKledelta}
            \Vol(\CC(x_t,K)) \le \delta.
        \end{equation}
        
        Assume that $\delta$ is small enough that 
        \begin{equation}\label{eq:tlessthanC0delta}
            t \le C_0\delta^{1/n},
        \end{equation}
        which is possible by Lemma~\ref{lem:Hausdorff}. Set
        \[
            C_3 = \max\left\{c^{-\frac{2}{n+1}},\,C_0^{\frac{2n}{n+1}}\right\}.
        \]
        If $t \le \rho(x_0)$ from \eqref{eq:volumecor} and \eqref{eq:VolCCxtKledelta} imply that $t$ belongs to the right-hand side of \eqref{eq:containmentcor}. If instead $t > \rho(x_0)$, then by \eqref{eq:tlessthanC0delta},
        \[
            t \le C_0\delta^{1/n} \le C_0^{\frac{2n}{n+1}}\delta^{\frac{2}{n+1}}\rho(x_0)^{-\frac{n-1}{n+1}},
        \]
        which again implies that $t$ belongs to the right-hand side of \eqref{eq:containmentcor}.
    \end{proof}
    
    \subsection{Step III: Integration}

    In this section we complete the proof of Theorem \ref{mainthm}. Recall the definition of the function $\rho : \Sigma \to [0,R]$ given in \eqref{eq:rhodef}.
    \begin{lemma}[\text{cf. \cite[Lemma 5]{SW}}]\label{rhointegrablelemma}
        For every $0 < \alpha < 1$, 
        \[
            \int_\Sigma\rho^{-\alpha}d\cH^{n-1} < \infty.
        \]
    \end{lemma}
    \begin{proof}
        Let $0 < \alpha < 1$. By Lemma \ref{lem:convexchart} and compactness, there exists a finite open cover $\{V_i\}$ of the hypersurface $\Sigma$ and smooth diffeomorphism $\{\Psi_i : V_i \to \tilde V_i \subseteq \RR^n\}$ such that for every $i$, the set $\Psi_i(K \cap V_i)$ is convex as a subset of $\RR^n$. 
        
        \medskip
        For every $i$ and every $y \in \Psi_i(\Sigma\cap V_i)$, let $\tilde\rho_i(y)$ denote the radius of the largest closed Euclidean sphere contained in $\Psi_i(K\cap V_i)$ and containing the point $y$. Then, by \cite[Lemma 5]{SW}, the function $\tilde \rho_i^{-\alpha}$ is $\cH^{n-1}$-integrable on $\Psi_i(\Sigma\cap V_i)$ for all $0 < \alpha <1$. Hence, since $\Psi_i$ is a diffeomorphism, if we set 
        \[
            \rho_i : = \tilde\rho_i \circ \Psi_i
        \]
        then $\rho_i^{-\alpha}$ is $\cH^{n-1}$-integrable on $\Sigma\cap V_i$. Since $\Sigma$ is covered by the finite collection $\{V_i\}$, if we set
        \[
            \hat \rho(x) : = \max\{\rho_i(x) \, : \, V_i \ni x\}, \qquad x \in \Sigma.
        \]
        then
        \[
            \int_\Sigma\hat\rho^{-\alpha}\,d\cH^{n-1} < \infty.
        \]
        Thus the proof will be finished once we show that there exists $c > 0$ such that $\rho \ge c \cdot \hat\rho$ on $\Sigma$.

        \medskip
        Let $x \in \Sigma$, and let $V_i \ni x$. We need to show that $\rho(x) \ge c \cdot \rho_i(x)$. By the definition of $\rho$, this means that the set $K$ contains a ball of radius $c \cdot \rho_i(x)$ such that the point $x$ lies on its boundary. If $\rho_i(x) = 0$, there is nothing to prove; assume therefore that $\rho_i(x) > 0$.
        
        \medskip
        By the definition of $\rho_i$, there exists a Euclidean ball $B$ of radius $\rho_i(x)$ contained in $\Psi_i(K \cap V_i)$ and containing $\Psi_i(x)$ on its boundary. Since $\{\Psi_i\}$ is a finite collection of smooth diffeomorphisms on relatively compact sets, there exists $c > 0$, independent of $i$, such that if $B'$ is a Riemannian ball of radius $c \rho_i(x)$ tangent to the hypersurface $\Psi_i^{-1}(\partial B)$ at $x$, then the \emph{Euclidean} second fundamental form of $\Psi_i(B')$ is everywhere larger than $\rho_i(x)^{-1}\cdot \mathrm{I}$. It then follows from Blaschke's rolling theorem \cite{Blaschke56} that $\Psi_i(B')\subseteq B$, and therefore $B' \subseteq \Psi_i^{-1}(B) \subseteq K$. 
    \end{proof}

    We can now finish the proof of Theorem \ref{mainthm}.
    By Lemma \ref{lem:Hausdorff},
    \[
        K^\delta \subseteq U_{C_0\delta^{1/n}}.
    \]
    Thus, by Proposition \ref{tubeprop},
    \begin{align*}
        \delta^{-\frac{2}{n+1}}\Vol(K^\delta\setminus K) & = \delta^{-\frac{2}{n+1}}\Vol\left((K^\delta\setminus K) \cap U_{C_0\delta^{1/n}}\right)\\
        & = \delta^{-\frac{2}{n+1}}\left(\int_{\Sigma'}\cH^1\left(\gamma_x^{-1}(K^\delta\setminus K)\right)d\cH^{n-1}(x) + O\left(\delta^{\frac2n}\right)\right)\\
        & = \int_{\Sigma'}\delta^{-\frac{2}{n+1}}\cdot \cH^1\left(\gamma_x^{-1}(K^\delta\setminus K)\right)d\cH^{n-1}(x) + O\left(\delta^{\frac{2}{n(n+1)}}\right).
    \end{align*}
    By Lemma \ref{lem:alexandrov}, $\cH^{n-1}(\Sigma'\setminus\Sigma'') = 0$, so in fact
    \begin{equation}\label{eq:finallimit}
        \delta^{-\frac{2}{n+1}}\Vol(K^\delta\setminus K) = \int_{\Sigma''}\delta^{-\frac{2}{n+1}}\cdot \cH^1\left(\gamma_x^{-1}(K^\delta)\right)d\cH^{n-1}(x) + O\left(\delta^{\frac{2}{n(n+1)}}\right).
    \end{equation}
    By Proposition \ref{volumeprop}, if $x \in \Sigma''$ satisfies $\kappa(x) > 0$ then
    \[
        \Vol(\CC(\gamma_x(t),K))^{\frac{2}{n+1}}
        =
        \beta_n^{-1}\cdot\kappa(x)^{-\frac{1}{n+1}}\cdot t + o(t),
    \]
    whence, by the definition of $K^\delta$, for every $\eps > 0$ there exists $t_0 = t_0(x,\eps) > 0$ such that
    \begin{equation}\label{eq:intervalsandwich}
        \left[0, (1 - \eps)\delta^{\frac{2}{n+1}}\beta_n\kappa(x)^{\frac{1}{n+1}}\right] \subseteq \gamma_x^{-1}(K^\delta\setminus K) \cap \left[0,t_0\right] \subseteq
        \left[0, (1 + \eps)\delta^{\frac{2}{n+1}}\beta_n\kappa(x)^{\frac{1}{n+1}}\right].
    \end{equation}
    
    By Lemma \ref{lem:Hausdorff}, $\gamma_x^{-1}(K^\delta) \subseteq [0,t_0]$ for every sufficiently small $\delta$. Combining this with \eqref{eq:intervalsandwich}, we get that for every $x \in \Sigma''$ with $\kappa(x) > 0$,
    \[
        \lim_{\delta\searrow 0}\delta^{-\frac{2}{n+1}}\cH^1\left(\gamma_x^{-1}(K^\delta\setminus K)\right) = \beta_n\kappa(x)^{\frac{1}{n+1}}.
    \]

    For $x \in \Sigma''$ with $\kappa(x) = 0$, Proposition~\ref{volumeprop} implies that for every $\eps > 0$ there exists $t_0 = t_0(\eps,x) > 0$ such that 
    \begin{equation}\label{eq:flatvolestimate}
        \Vol(\CC(\gamma_x(t),K))^{2/(n+1)} \ge t/\eps \qquad \text{ for all } 0 < t < t_0.
    \end{equation}
    By Lemma \ref{lem:Hausdorff}, if $\delta$ is sufficiently small then $\gamma_x^{-1}(K^\delta) \subseteq [0, t_0]$, and therefore \eqref{eq:flatvolestimate} gives
    \[
        \delta^{-2/(n+1)}\cH^1(\gamma_x^{-1}(K^\delta\setminus K)) \le \eps.
    \]
    Thus $\delta^{-2/(n+1)}\cH^1(\gamma_x^{-1}(K^\delta\setminus K)) \to 0$.

    \medskip
    In summary,
    \begin{equation}\label{eq:pointwiselim}
        \lim_{\delta\searrow 0}\delta^{-\frac{2}{n+1}}\cH^1\left(\gamma_x^{-1}(K^\delta\setminus K)\right) = \beta_n\kappa(x)^{\frac{1}{n+1}} \qquad \text{ for every $x \in \Sigma''$}.
    \end{equation}
    By Lemma \ref{lem:lowerbound}, for $\cH^{n-1}$-almost every $x \in \Sigma'$,
    \[
        \cH^1\left(\gamma_x^{-1}(K^\delta\setminus K)\right) \le C_3 \cdot \delta^{\frac{2}{n+1}}\cdot \rho(x)^{-\frac{n-1}{n+1}},
    \]
    and therefore, by Lemma \ref{rhointegrablelemma}, the integrand on the right hand side of \eqref{eq:finallimit} is bounded by an integrable function, and by \eqref{eq:pointwiselim} it converges pointwise to $\beta_n\cdot\kappa^{\frac{1}{n+1}}$.
    We may therefore apply the dominated convergence theorem and obtain formula \eqref{maineq}. The proof of Theorem \ref{mainthm} is complete.

    \appendix

    \section{Appendix}

    \subsection{Approximate Euclidean geometry on Riemannian manifolds}
    This appendix contains some elementary facts about Riemannian geometry, specifically its ``closeness'' to Euclidean geometry in small scales. Throughout, $(M,g)$ will be a Riemannian manifold and $A \subseteq M$ a compact, strongly convex set; all constants implied by big-O notation are allowed to depend only on the set $A$.
    
    \begin{lemma}\label{normallemma}
        Let $x \in A$ and $R > 0$ and assume that the ball $B_R(x)$ is contained in the set $A$. Let $L: \RR^n \to T_xM$ be a linear isometry and define $G : B_R(0) \to B_R(x)$ by $G : = \exp_x\circ \, L$. Then
        \begin{equation}\label{eq:normal}
            (G^*g)_{ij}\vert_v = \delta_{ij} + O(|v|^2) \qquad \text{ and } \qquad \sqrt{\det(G^*g\vert_v)} = 1 + O(|v|^2)
        \end{equation}
        for all $v \in B_R(0)$. Consequently, if $d_0$ and $\Vol_0$ denote the Euclidean distance and volume on $\RR^n$, respectively, then for every $v,v' \in B_R(0)$,
        \begin{equation}\label{eq:dd0}
            d(G(v),G(v')) = d_0(v,v')(1 + O(R^2)),
        \end{equation}
        and for every Borel set $S \subseteq B_R(0)$,
        \begin{equation}\label{eq:VolVol0}
            \Vol(G(S)) = \Vol_0(S)(1 + O(R^2)).
        \end{equation}
    \end{lemma}
    \begin{proof}
        The estimates in \eqref{eq:normal} can be proved using elementary Jacobi field computations, see e.g. \cite[Problem 10-1]{Lee}. Estimate \eqref{eq:dd0} can be derived by integrating the first estimate in \eqref{eq:normal} along a straight line joining a pair of points in $B_R(0)$, and then along a minimizing geodesic of $g$ joining the same pair of points. Estimate \eqref{eq:VolVol0} follows from integrating the second estimate in \eqref{eq:normal} on the set $S$.
    \end{proof}

    \begin{lemma}\label{lem:normal-geodesic-estimate}
        Let $x \in A$ and $R > 0$ and assume that the ball $B_R(x)$ is contained in the set $A$. Let $\gamma:[0,1]\to B_R(x)$ be a constant-speed geodesic. Then in normal coordinates centered at $x$,
        \begin{enumerate}[$(i)$]
            \item \label{normal-geodesic-initial}
            \(
                \gamma(s)=\gamma(0)+s\dot\gamma(0)+O(s^2|\dot\gamma(0)|^2) \qquad \text{and}
            \)
            \item \label{normal-geodesic-chord}
            \(
                \gamma(s)=(1-s)\gamma(0)+s\gamma(1)+O(|\gamma(1)-\gamma(0)|^2), \qquad 0\le s\le 1.
            \)
        \end{enumerate}
    \end{lemma}
    \begin{proof}
        By compactness, the Christoffel symbols in normal coordinates on $B_R(0)$ are bounded by a constant depending only on $A$. Hence the geodesic equation gives $\ddot\gamma=O(|\dot\gamma|_0^2)$, where $|\cdot|_0$ is the Euclidean norm and the second derivative is taken coordinate-wise. Since the Riemannian speed of $\gamma$ is constant, by Lemma \ref{normallemma} we have $|\dot\gamma|_0=O(|\dot\gamma(0)|)$. Integrating twice gives \eqref{normal-geodesic-initial}. By \eqref{eq:dd0} we also have $|\dot\gamma(0)|_0=O(|\gamma(1)-\gamma(0)|)$, whence $\dot\gamma(0)=\gamma(1)-\gamma(0)+O(|\gamma(1)-\gamma(0)|^2)$. Substituting this into \eqref{normal-geodesic-initial} with $s=1$ gives \eqref{normal-geodesic-chord}.
    \end{proof}

    For $r > 0$ and $k \in \RR$ we define
    \[
        \cos_k(r):=
        \begin{cases}
            \cos(\sqrt{k}\,r), & k>0,\\
            1 & k = 0,\\
            \cosh(\sqrt{-k}\,r), & k<0
        \end{cases}
        \qquad
        \text{and}
        \qquad
        \sin_k(r):=
        \begin{cases}
            \frac{\sin(\sqrt{k}\,r)}{\sqrt k}, & k>0,\\
            r & k = 0,\\
            \frac{\sinh(\sqrt{-k}\,r)}{\sqrt{-k}}, & k<0.
        \end{cases}
    \]

    \begin{lemma}\label{hessianlemma}
    Let $x \in A$ and denote by $r_x : = d(\cdot,x)$ the distance function to the point $x$. Then
    \[
        \Hess r_x = \frac{1}{r_x} \cdot \left(g - (dr_x)^2\right) + O(r_x) \qquad
        \text{on the set $A \setminus \{x\}$}.
    \]
    \end{lemma}
    \begin{proof}
        If $k_-$ and $k_+$ are lower and upper bounds on the sectional curvature of $M$ on the set $A$, respectively, then by Hessian comparison \cite[Theorem 11.7]{Lee},
        \[
            \frac{\cos_{k_+}(r_x)}{\sin_{k_+}(r_x)} \cdot \left(g - (dr_x)^2\right) 
            \le
             \Hess r_x \le \frac{\cos_{k_-}(r_x)}{\sin_{k_-}(r_x)} \cdot \left(g - (dr_x)^2\right).
        \]
        The lemma now follows from the estimate $\frac{\cos_k(r)}{\sin_k(r)} = \frac1r + O(r)$ for $k \in \RR$, with the implied constant depending on $k$.
    \end{proof}

    \begin{lemma}\label{lem:smalltriangle}
        Let $\triangle xyz$ be a geodesic triangle contained in the set $A$. Write
        \[
            a : = d(x,y), \quad b : = d(x,z), \quad c : = d(y,z) \quad \text{ and } \quad \theta: = \angle yxz \in (0,\pi).
        \]
        Assume that
        \[
            \max\{a,b\} \le \frac{\pi}{2\sqrt{k_{\max}}},
        \]
        where $k_{\max}$ is the maximum over all sectional curvatures in the set $A$ (if the sectional curvature is nonpositive on $A$, then we make no assumption on $\max\{a,b\}$). Then
        \begin{equation}\label{eq:coscomparison}
            \frac{\cos_{k_+}(c) - \cos_{k_+}(a)\cos_{k_+}(b)}{k_+\sin_{k_+}(a)\sin_{k_+}(b)} \le \cos\theta \le \frac{\cos_{k_-}(c) - \cos_{k_-}(a)\cos_{k_-}(b)}{k_-\sin_{k_-}(a)\sin_{k_-}(b)}.
        \end{equation}
        where $k_-$ (resp. $k_+$) is a lower (resp. upper) bound on the sectional curvature on the set $A$. When $k = 0$ we understand $(\cos_kc - \cos_ka\cos_kb)/(k\sin_ka\sin_kb)$ via a limit to be $(a^2+b^2-c^2)/2ab$.
        
        \medskip
        In particular, if $\theta = \pi/2$ then 
        \[
            \arccos_{k_+}(\cos_{k_+}(a)\cos_{k_+}(b)) \le c \le \arccos_{k_-}(\cos_{k_-}(a)\cos_{k_-}(b)),
        \]
        where 
        \[
        \arccos_k(r):=
        \begin{cases}
            \frac{\arccos(r)}{\sqrt{k}} & k > 0\\
            \frac{\mathrm{arccosh}(r)}{\sqrt{-k}} & k < 0,
        \end{cases}
        \]
        and when $k = 0$ we understand the expression $\arccos_0(\cos_0(a)\cos_0(b))$ via a limit to be $\sqrt{a^2 + b^2}$. 
    \end{lemma}

    \begin{proof}
        By the Rauch and Toponogov comparison theorems \cite[Section 4]{Karch}, if for $k \in \RR$ we denote by $\tilde\theta_k$ the angle $\angle\tilde y\tilde x\tilde z$ in a comparison triangle $\triangle \tilde x \tilde y \tilde z$ on a space of constant curvature $k$ satisfying $d(\tilde x,\tilde y) = a$, $d(\tilde x,\tilde z) = b$ and $d(\tilde y,\tilde z) = c$, then
        \[
            \widetilde\theta_{k_-}\le \theta \le \widetilde\theta_{k_+},
        \]
        whence
        \[
            \cos\widetilde\theta_{k_-}
            \ge
            \cos\theta
            \ge
            \cos\widetilde\theta_{k_+}.
        \]
        By the law of cosines in constant curvature,
        \[
            \cos_k(c)
            =
            \cos_k(a)\cos_k(b)
            +
            k\sin_k(a)\sin_k(b)\cos\widetilde\theta_k,
        \]
        so \eqref{eq:coscomparison} follows. 
    \end{proof}    

    \begin{lemma}\label{righttrianglelemma}
        Let $\triangle xyz$ be a right geodesic triangle, $\angle xyz = \pi/2$, contained in the set $A$. Denote 
        \[
            a : = d(x,y), \quad b:=d(y,z), \quad c : = d(x,z) \quad \text{ and } \quad \vphi:=\angle zxy,
        \]
        Assume that
        \[
            l : = \max\{a,b,c\} \le \frac{\pi}{2\sqrt{k_{\max}}},
        \]
        where $k_{\max}$ is the maximum over all sectional curvatures in the set $A$ (if the sectional curvature is nonpositive on $A$, then we make no assumption on $l$). Then
        \begin{equation}\label{approxtrigeq}
            b = c \cdot \sin\vphi \cdot \left(1 + O(l^2)\right),
            \qquad a = c \cdot \cos\vphi \cdot \left(1 + O(l^2)\right),
        \end{equation}
        and 
        \begin{equation}\label{pytheq}
            c = \sqrt{a^2 + b^2} + O((a^2+b^2)^{3/2}) = \sqrt{a^2 + b^2} + O(l^3).
        \end{equation}
    \end{lemma}

    \begin{proof}
        Estimate \eqref{pytheq} follows from Lemma \ref{normallemma} by passing to normal coordinates at $y$, where the two initial tangent vectors are orthogonal. On the other hand, applying \eqref{eq:dd0} in normal coordinates at $x$, with $v=\exp_x^{-1}(y)$ and $v'=\exp_x^{-1}(z)$, gives
        \[
            b^2=a^2+c^2-2ac\cos\vphi+O(l^4).
        \]
        Combining this with \eqref{pytheq}, equivalently with $c^2=a^2+b^2+O(l^4)$, gives
        \[
            \cos\vphi=\frac{a}{c}+O(l^2).
        \]
        Hence $a=c\cos\vphi(1+O(l^2))$. Substituting this estimate in \eqref{pytheq} gives $b=c\sin\vphi(1+O(l^2))$, and so \eqref{approxtrigeq} follows.
    \end{proof}

    \subsection{Second-order differentiability of the signed distance function}

    \begin{lemma}\label{lem:sdf}
        Let $(M,g)$ be a Riemannian manifold, let $K \subseteq M$, let $y \in \partial K$ and suppose that there exists a neighborhood $V \ni y$ and a function $\psi :V \to \RR$ such that
        \begin{equation}\label{eq:psidefining}
            K \cap V = \psi^{-1}((-\infty,0])
        \end{equation}
        and such that $\psi$ is differentiable twice at $y$ in the sense of Alexandrov with $d\psi\vert_y \ne 0$. Then the signed distance function $r$ to $\partial K$ is differentiable twice at $y$ in the sense of Alexandrov.
    \end{lemma}

    \begin{figure}[ht]
        \centering
        \begin{tikzpicture}[scale=1.7]

            \begin{scope}
                \clip (0,0) circle (1.4);

                \fill[gray!25]
                    (-2,-2) --
                    plot[domain=-2:2, smooth] (\x, {-0.2*\x*\x - 0.25}) --
                    (2,-2) -- cycle;

                \fill[gray!10]
                    plot[domain=-2:2, smooth] (\x, {-0.2*\x*\x + 0.25}) --
                    plot[domain=2:-2, smooth] (\x, {-0.2*\x*\x - 0.25}) --
                    cycle;

                \draw[thick]
                    plot[domain=-2:2, smooth] (\x, {-0.2*\x*\x});

                \draw[dashed, thick]
                    plot[domain=-2:2, smooth] (\x, {-0.2*\x*\x + 0.25});

                \draw[dashed, thick]
                    plot[domain=-2:2, smooth] (\x, {-0.2*\x*\x - 0.25});
            \end{scope}

            \draw[gray] (0,0) circle (1.4);

            \draw[dashed, opacity=0.5] (0,0) -- (-0.99, -0.99)
                node[midway, below] {\small $\delta$};
            \node at (0, 0.7) {\small $B_\delta(0)\setminus P_{\varepsilon\delta^2}$};

            \node[right] at (1.42, -0.37) {\small $\partial K$};

            \fill (0,0) circle (1.5pt);
            \node[below right] at (0,0) {$0$};

            \node at (0,-0.85) {$P_{-\varepsilon\delta^2}$};
        \end{tikzpicture}
        \caption{Proof of Lemma \ref{lem:sdf}}
        \label{fig:sdf}
    \end{figure}
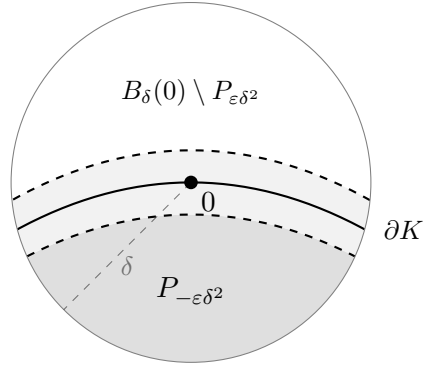

    \begin{proof}
        The notion of second order differentiability (in the sense of Alexandrov, we omit these extra words in the sequel) is invariant under diffeomorphisms, so we may work in normal coordinates centered at $y$; the function $\psi$ is then differentiable twice at the origin, and our goal is to prove that the function $r$ is differentiable twice at the origin.

        \medskip
        Let $f$ denote the second-order Taylor polynomial of the function $\psi$ at $0$:
        \[
            f(x) : = d\psi\vert_0(x) + \frac12 \cdot \Hess\psi\vert_0(x),
        \]
        and write
        \[
            P_t : = \left\{x \in \RR^n \, : \,  f(x) \le t\right\}, \qquad  t \in \RR.
        \]
        Since $d\psi\vert_0 \ne 0$, for sufficiently small $\delta$ the Euclidean Hausdorff distance between $\partial P_t \cap B_\delta(0)$ and $\partial P_0 \cap B_\delta(0)$ is $O(t)$. Hence if we denote by $r_t$ the Euclidean signed distance function to $\partial P_t$ then 
        \begin{equation}\label{eq:rtdelta}
            r_t = r_0 + O(t) \qquad \text{on $B_\delta(0)$}.
        \end{equation}

        \medskip
        We claim that
        \begin{equation}\label{eq:rr0delta}
            r = r_0 + o(\delta^2) \qquad \text{on $B_{\delta/4}(0)$.}
        \end{equation}
        Since $r_0$ is the signed distance function to $\{f=0\}$, which is smooth near $0$, this will prove that $r$ is differentiable twice at $0$.

        \medskip
        Let $\eps > 0$. Since $\psi$ is differentiable twice at $0$, for all sufficiently small $\delta$ we have
        \[
            |\psi(x) - f(x)| \le \eps\delta^2 \qquad \text{ for all $x \in B_\delta(0)$},
        \]
        whence by \eqref{eq:psidefining},
        \begin{equation}\label{eq:KcapBdelta}
            P_{-\eps\delta^2} \cap B_\delta(0) \subseteq K \cap B_\delta(0) \subseteq P_{\eps\delta^2} \cap B_\delta(0),
        \end{equation}
        see Figure \ref{fig:sdf}. Denote Euclidean distance by $d_0$ (Riemannian distance is denoted by $d$ as usual). Since $0 \in \partial K$ and $f(0) = 0$, for $x \in B_{\delta/4}(0)$ the nearest points to $x$ of $K$, of $M \setminus K$, and of each hypersurface $\{f=t\}$ with $|t| \le \eps\delta^2$ all lie in $B_\delta(0)$ (their distance from $0$ is at most $2|x| + O(\eps\delta^2) < \delta$). In particular,
        \begin{equation}\label{eq:distid}
            d_0(x,B_\delta(0)\setminus P_t) = |r_t(x)| \qquad \text{for $|t| \le \eps\delta^2$ and $x \in P_t \cap B_{\delta/4}(0)$.}
        \end{equation}

        \medskip
        Let $x \in B_{\delta/4}(0)$, and consider three cases:
        \begin{itemize}
            \item If $x \in P_{-\eps\delta^2}$, then by \eqref{eq:KcapBdelta}, Lemma \ref{normallemma} and \eqref{eq:rtdelta},
            \begin{align*}
                -r(x) & = d(x,M\setminus K) \\
                & \le d\left(x,B_\delta(0)\setminus P_{\eps\delta^2}\right)\\
                & = d_0(x,B_\delta(0)\setminus P_{\eps\delta^2}) + O(\delta^3) \\
                & = - r_{\eps\delta^2}(x) + O(\delta^3)\\
                & = - r_0(x) + O(\eps\delta^2 + \delta^3)
            \end{align*}
            and
            \begin{align*}
                -r(x) & = d(x,M\setminus K) \\
                & \ge d\left(x,B_\delta(0)\setminus P_{-\eps\delta^2}\right)\\
                & = d_0(x,B_\delta(0)\setminus P_{-\eps\delta^2}) + O(\delta^3)\\
                & = -r_{-\eps\delta^2}(x) + O(\delta^3)\\
                & = -r_0(x) + O(\eps\delta^2 + \delta^3).
            \end{align*}
            \item If $x \notin P_{\eps\delta^2}$, then $x \notin K$ and therefore $r(x) = d(x,K)$. By a similar argument,
            \[
                r(x) = r_0(x) + O(\eps\delta^2 + \delta^3).
            \]
            \item Finally, if $x \in P_{\eps\delta^2}\setminus P_{-\eps\delta^2}$, then by \eqref{eq:distid} and \eqref{eq:rtdelta} ,
            \begin{align*}
                |r(x)| & \le \max\left\{d(x,K),d(x,B_\delta(0)\setminus K)\right\}\\ 
                & \le \max\left\{d(x, P_{-\eps\delta^2}),d(x,B_\delta(0)\setminus P_{\eps\delta^2})\right\}\\
                & = \max\left\{d_0(x, P_{-\eps\delta^2}),d_0(x,B_\delta(0)\setminus P_{\eps\delta^2})\right\} + O(\delta^3)\\
                &  = |r_0(x)| + O(\eps\delta^2 + \delta^3).
            \end{align*}
            
            Since $x$ lies between $\{f=-\eps\delta^2\}$ and $\{f=\eps\delta^2\}$, we have $|r_0(x)| = O(\eps\delta^2)$, whence
            \[
                |r(x) - r_0(x)| \le |r(x)| + |r_0(x)| = O(\eps\delta^2 + \delta^3).
            \]
        \end{itemize}
        Since $\eps$ is arbitrary, this finishes the proof of \eqref{eq:rr0delta} and of the lemma.  
    \end{proof}

    \bibliographystyle{plain}
    \bibliography{references}

\begin{thebibliography}{10}

\bibitem{ABW}
Rotem Assouline, Florian Besau, and Elisabeth~M. Werner.
\newblock Illumination bodies in projective geometries.
\newblock {\em Submitted, arXiv 2605.25122}, 2026.

\bibitem{BaranyLarman}
Imre B{\'a}r{\'a}ny and David~G. Larman.
\newblock Convex bodies, economic cap coverings, random polytopes.
\newblock {\em Mathematika}, 35:274--291, 1988.

\bibitem{BW:2016}
Florian Besau and Elisabeth~M. Werner.
\newblock The spherical convex floating body.
\newblock {\em Adv. Math.}, 301:867--901, 2016.

\bibitem{BW}
Florian Besau and Elisabeth~M. Werner.
\newblock The floating body in real space forms.
\newblock {\em J. Differential Geom.}, 110(2):187--220, 2018.

\bibitem{BCP}
Gabriele Bianchi, Andrea Colesanti, and Carlo Pucci.
\newblock On the second differentiability of convex surfaces.
\newblock {\em Geom. Dedicata}, 60(1):39--48, 1996.

\bibitem{Blaschke}
Wilhelm Blaschke.
\newblock {\em Differentialgeometrie {II}, {A}ffine {D}ifferentialgeometrie}.
\newblock Springer-Verlag, Berlin, 1923.

\bibitem{Blaschke56}
Wilhelm Blaschke.
\newblock {\em Kreis und {K}ugel}.
\newblock Walter de Gruyter \& Co., Berlin, 1956.
\newblock 2te Aufl.

\bibitem{Chavel}
Isaac Chavel.
\newblock {\em Riemannian geometry. {A} modern introduction}.
\newblock Number~98 in Camb. Stud. Adv. Math. Cambridge University Press,
  Cambridge, second edition, 2006.

\bibitem{CE}
Jeff Cheeger and David~G. Ebin.
\newblock {\em Comparison theorems in {R}iemannian geometry}.
\newblock AMS Chelsea Publishing, Providence, RI, 2008.
\newblock Revised reprint of the 1975 original.

\bibitem{CG}
Jeff Cheeger and Detlef Gromoll.
\newblock On the structure of complete manifolds of nonnegative curvature.
\newblock {\em Ann. of Math. (2)}, 96:413--443, 1972.

\bibitem{Dupin}
Charles Dupin.
\newblock {\em Application de g{\'e}ometrie et de m{\'e}chanique {\`a} la
  marine, aux ponts et chauss{\'e}es}.
\newblock Bachelier, Paris, 1822.

\bibitem{EG}
Lawrence~C. Evans and Ronald~F. Gariepy.
\newblock {\em Measure theory and fine properties of functions}.
\newblock Textb. Math. CRC Press, Boca Raton, FL, revised edition, 2015.

\bibitem{Gr}
Alfred Gray.
\newblock {\em Tubes}.
\newblock Number 221 in Progr. Math. Birkh{\"a}user Verlag, Basel, second
  edition, 2004.

\bibitem{Karch}
Hermann Karcher.
\newblock Riemannian comparison constructions.
\newblock In {\em Global differential geometry}, volume~27 of {\em MAA Stud.
  Math.}, pages 170--222. Math. Assoc. America, Washington, DC, 1989.

\bibitem{Lee}
John~M. Lee.
\newblock {\em Introduction to {R}iemannian manifolds}.
\newblock Number 176 in Grad. Texts in Math. Springer, Cham, second edition,
  2018.

\bibitem{Lutwak}
Erwin Lutwak.
\newblock The {B}runn-{M}inkowski-{F}irey theory. {II}. {A}ffine and geominimal
  surface areas.
\newblock {\em Adv. Math.}, 118:244--294, 1996.

\bibitem{Lyt}
Alexander Lytchak.
\newblock A note on subsets of positive reach.
\newblock {\em Math. Nachr.}, 297(3):932--942, 2024.

\bibitem{Reitzner}
Matthias Reitzner.
\newblock Random points on the boundary of smooth convex bodies.
\newblock {\em Trans. Amer. Math. Soc.}, 354(6):2243--2278, 2002.

\bibitem{Sa}
L.~A. Santal{\'o}.
\newblock On parallel hypersurfaces in the elliptic and hyperbolic
  $n$-dimensional space.
\newblock {\em Proc. Amer. Math. Soc.}, 1:325--330, 1950.

\bibitem{Sc}
Rolf Schneider.
\newblock {\em Convex bodies: the {B}runn-{M}inkowski theory}.
\newblock Number 151 in Encyclopedia Math. Appl. Cambridge University Press,
  Cambridge, second expanded edition, 2014.

\bibitem{SW}
Carsten Sch{\"u}tt and Elisabeth Werner.
\newblock The convex floating body.
\newblock {\em Math. Scand.}, 66(2):275--290, 1990.

\bibitem{SW5}
Carsten Sch{\"u}tt and Elisabeth Werner.
\newblock Polytopes with vertices chosen randomly from the boundary of a convex
  body.
\newblock In {\em Geometric aspects of functional analysis}, volume 1807 of
  {\em Lecture Notes in Math.}, pages 241--422. Springer, Berlin, 2003.

\bibitem{ST}
Gil Solanes and Juan~Andr{\'e}s Trillo.
\newblock Tube formulas for valuations in complex space forms.
\newblock {\em Math. Ann.}, 391(1):881--913, 2025.

\bibitem{TW:2023}
Kateryna Tatarko and Elisabeth~M. Werner.
\newblock $l_p$-{S}teiner quermassintegrals.
\newblock {\em Adv. Math.}, 430, 2023.

\bibitem{Vil}
C\'edric Villani.
\newblock {\em Optimal transport}, volume 338 of {\em Grundlehren der
  mathematischen Wissenschaften [Fundamental Principles of Mathematical
  Sciences]}.
\newblock Springer-Verlag, Berlin, 2009.
\newblock Old and new.

\bibitem{Wal74}
Rolf Walter.
\newblock On the metric projection onto convex sets in {R}iemannian spaces.
\newblock {\em Arch. Math. (Basel)}, 25:91--98, 1974.

\bibitem{Wal}
Rolf Walter.
\newblock Some analytical properties of geodesically convex sets.
\newblock {\em Abh. Math. Sem. Univ. Hamburg}, 45:263--282, 1976.

\bibitem{Wer}
Elisabeth Werner.
\newblock Illumination bodies and affine surface area.
\newblock {\em Studia Math.}, 110(3):257--269, 1994.

\bibitem{We}
Hermann Weyl.
\newblock On the volume of tubes.
\newblock {\em Amer. J. Math.}, 61(2):461--472, 1939.

\bibitem{Ye:2015}
Deping Ye.
\newblock New {O}rlicz affine isoperimetric inequalities.
\newblock {\em J. Math. Anal. Appl.}, 427:905--929, 2015.

\bibitem{Ye:2016}
Deping Ye.
\newblock Dual {O}rlicz-{B}runn-{M}inkowski theory: dual {O}rlicz $l_\phi$
  affine and geominimal surface areas.
\newblock {\em J. Math. Anal. Appl.}, 443:352--371, 2016.

\bibitem{Zhao2016}
Yiming Zhao.
\newblock On $l_p$-affine surface area and curvature measures.
\newblock {\em Int. Math. Res. Not. IMRN}, 2016(5):1387--1423, 2016.

\end{thebibliography}
\end{document}